\documentclass{amsart}
\usepackage{amssymb}
\usepackage{amscd}
\usepackage[latin1]{inputenc}
\usepackage{amsmath}
\usepackage{amsfonts}
\usepackage{latexsym}
\usepackage{graphicx}
\usepackage{color}
\usepackage{epsfig}
\newtheorem{theorem}{Theorem}[section]
\newtheorem{corollary}[theorem]{Corollary}
\newtheorem{lemma}[theorem]{Lemma}
\newtheorem{definition}[theorem]{Definition}

\newtheorem{proposition}[theorem]{Proposition}
\newtheorem{example}[theorem]{Example}
\newtheorem{remark}[theorem]{Remark}
\newtheoremstyle{nonum}{}{}{\upshape}{}{\itshape}{.}{ }{#1#3}
\theoremstyle{nonum}
\newtheorem{remark*}{Remark}

\newcommand{\comment}[1]{}
\def \beq {\begin{eqnarray}}
\def \eeq {\end{eqnarray}}
\def \beqn {\begin{eqnarray*}}
\def \eeqn {\end{eqnarray*}}

\def\A{{\mathcal A }}

\def\F{{\mathcal F }}

\def\L{{\mathcal L }}

\def\mN{{\mathbb{N}}}
\def\mZ{{\mathbb{Z}}}
\def\mR{{\mathbb{R}}}
\def\zz{{\mathbb{Z}}}
\def\PP{{\mathcal{P}}}
\def\supp{{\mathrm{supp}}}
\def\nn{{nearest-neighbor}}

\begin{document}
\title[Representing topological pressure via conditional probabilities]{An integral representation for topological pressure in terms of conditional probabilities}

\begin{abstract}
 Given an equilibrium state $\mu$ for a continuous function $f$ on a shift of finite type $X$, the pressure of $f$ is the integral, with respect to $\mu$, of the sum of $f$ and the information function of $\mu$. We show that under certain assumptions on $f$, $X$ and an invariant measure $\nu$, the pressure of $f$ can also be represented as the integral with respect to $\nu$ of the same integrand.
 Under stronger hypotheses we show that
this representation holds for all invariant measures $\nu$.  We establish an algorithmic implication for approximation of pressure, and we relate our results to a result in thermodynamic formalism.
\end{abstract}

\date{}
\author{Brian Marcus}
\address{Brian Marcus\\
Department of Mathematics\\
University of British Columbia\\
1984 Mathematics Road\\
Vancouver, BC V6T 1Z2}
\email{marcus@math.ubc.ca}
\author{Ronnie Pavlov}
\address{Ronnie Pavlov\\
Department of Mathematics\\
University of Denver\\
2360 S. Gaylord St.\\
Denver, CO 80208} \email{rpavlov@du.edu} 
\renewcommand{\subjclassname}{MSC 2000}
\maketitle

\section{Introduction}\label{intro}

Given a finite alphabet $\A$, the entropy of a shift-invariant
measure $\mu$ on $\A^{\mZ}$ is sometimes defined as the expected
conditional entropy of the present given the past, or
\[
h(\mu) = \int H(x_0 \ | \ x_{-1}, x_{-2}, \ldots) \ d\mu(x_{-1},
x_{-2}, \ldots),
\]
where
\[
 H(x_0 \ | \ x_{-1}, x_{-2}, \ldots)  =
 \sum_{a \in \A} - \mu(x_0 = a \ | \ x_{-1}, x_{-2}, \ldots) \log
\mu(x_0 = a \ | \ x_{-1}, x_{-2}, \ldots).
\]
Equivalently,
\[
h(\mu) = \int - \log \mu(x_0 \ | \ x_{-1}, x_{-2}, \ldots) \
d\mu(x).
\]
It is less well-known that there is an analogue of this formula for
a shift-invariant measure $\mu$ on $\A^{\mZ^d}$ which involves the
notion of lexicographic past. Define $\mathcal{P} \subseteq
\mathbb{Z}^d$ to be the set of sites lexicographically less than
$0$, or equivalently the set of $v \in \mathbb{Z}^d \setminus \{0\}$
whose last  nonzero coordinate is negative. Then
\begin{equation}\label{entpast}
h(\mu) = \int - \log \mu(x_0 \ | \ \{x_p\}_{p \in \mathcal{P}}) \
d\mu
\end{equation}
(see~\cite[Theorem 15.12]{Georg} or ~\cite[p. 283, Theorem
2.4]{Krengel}).

The integrand in (\ref{entpast}), which we will denote by
$I_{\mu}(x)$, is fundamental in ergodic theory and information
theory and is known as the information function.

We will mostly be interested in implications of (\ref{entpast}) on
topological pressure. For any continuous function $f$ on a
$\mathbb{Z}^d$ shift of finite type, it is a consequence of the
variational principle (\cite{walters_var},~\cite{walters}) that
the measure-theoretic pressure function $P_{\rho}(f) = h(\rho) +
\int f \ d\rho$ achieves its maximum on a nonempty set of measures;
such measures are called equilibrium states of $f$, and the maximum
is called the topological pressure $P(f)$. When the function $f$ is locally
finite, all such equilibrium states are examples of so-called finite-range Gibbs
measures (\cite{LR},~\cite{Ruelle}). 
For more on equilibrium states and Gibbs measures, see Sections~\ref{MRF}
and~\ref{pressure}.

When applied to  an equilibrium state $\mu$ for $f$, (\ref{entpast})
clearly implies
\begin{equation}\label{genpressurerep}
P(f) = \int I_{\mu}(x) + f(x) \ d\mu.
\end{equation}

For certain classes of
equilibrium states and Gibbs measures, sometimes there are even
simpler representations for the pressure.
A recent example of this was given by Gamarnik and Katz in
\cite[Theorem 1]{GK}, who showed that for any Gibbs measure $\mu$
which has a measure-theoretic mixing property, called strong spatial
mixing, and whose support contains a so-called safe symbol $0$ (a
very strong topological mixing property, defined in Section~\ref{subshifts}),
\begin{equation}\label{safepressurerep}
P(f) = I_{\mu}(0^{\mathbb{Z}^d}) + f(0^{\mathbb{Z}^d})
\end{equation}
(here, $0^{\mathbb{Z}^d} \in \A^{\zz^d}$ is the configuration on
$\zz^d$ which is 0 at every site of $\zz^d$). This result was the
primary motivation for our paper.

They used this simple representation to give a polynomial time
approximation algorithm for $P(f)$ in certain cases. Approximation
schemes are very important because in most cases it is quite
difficult (and sometimes impossible!) to obtain exact, closed form
expressions for entropy and pressure (\cite{HM}), let alone the exact values of
conditional $\mu$-measures of specific cylinder sets or exact values
of the function $I_{\mu}$ that are needed to evaluate
(\ref{safepressurerep}).

A consequence (Corollary~\ref{easycor}; see also
Corollary~\ref{easycor3}) of one of our main results is that under
certain hypotheses on the equilibrium state $\mu$ and its support, the
integrand in (\ref{genpressurerep}) yields $P(f)$ when integrated
against \textit{any} shift-invariant measure $\nu$ whose support
is contained within the support of $\mu$:
\begin{equation}\label{nupressurerep}
P(f) = \int I_{\mu}(x) + f(x) \ d\nu.
\end{equation}

For instance, (\ref{safepressurerep}) is the special case of
(\ref{nupressurerep}) where $\nu$ is the point mass at
$0^{\mathbb{Z}^d}$. This consequence is related to a result of
Ruelle~\cite[4.7b]{Ruelle} which characterizes when two interactions
have a common Gibbs measure; this connection is discussed in
Section~\ref{connections}. The conclusion of Corollary~\ref{easycor}
is well known when $d=1$ and the support of $\mu$ is an irreducible
shift of finite type.

We have two main results, Theorems~\ref{easythm} and ~\ref{hardthm},
with different sets of hypotheses for representation of pressure of
$f$ with respect to a given invariant measure $\nu$.
For each theorem, our hypotheses are of three broad types: a topological mixing
condition on the support of $\mu$ (see Section~\ref{D-condition}), a type of continuity assumption
on $p_{\mu}$ over $\supp(\nu)$ (see Section~\ref{SSM}), and a type of positivity assumption on $p_{\mu}$
over $\supp(\nu)$ (see Section~\ref{positivity}).
We state and prove Theorem~\ref{easythm} for arbitrary dimension $d$, but
Theorem~\ref{hardthm} only for dimension $2$; we believe that Theorem~\ref{hardthm} is
true for all $d$ but the construction involved in the proof for $d >2$ is somewhat intricate geometrically.

All of these hypotheses, for both
theorems, are weaker than those used in \cite[Theorem 1]{GK}, and so
each implies (\ref{safepressurerep}) by taking $\nu$ to be the
point mass at $0^{\mathbb{Z}^d}$.

From Theorem~\ref{easythm} we prove Proposition~\ref{approx_scheme}
which, for $d > 1$ and certain $f$, shows that $P(f)$ can be approximated to within
tolerance $\epsilon$ in time $e^{O((\log \frac{1}{\epsilon})^{d-1})}$.
This yields a polynomial time approximation scheme when $d=2$ and
generalizes some cases of the approximation algorithms given
in~\cite{GK}. This scheme is more efficient than the approximation scheme
of~\cite{MP2}, but requires the additional hypothesis that
(\ref{nupressurerep}) holds for some ``simple'' $\nu$.

We summarize the remainder of this paper. In Section~\ref{defns}, we
give relevant definitions, background and preliminary results. In
Section~\ref{main}, we prove our main results,
Theorems~\ref{easythm} and ~\ref{hardthm}. In
Section~\ref{approx_alg}, we describe a consequence of our pressure
representation results that yields a method for approximating
pressure, and we discuss connections with work of Gamarnik and Katz
\cite{GK}. Finally, in Section~\ref{connections}, we describe a
connection between our results and the thermodynamic formalism of
Ruelle \cite{Ruelle}.

\section{Definitions and Preliminary Results}
\label{defns}

\subsection{Subshifts}
\label{subshifts}
~

We view $\mathbb{Z}^d$ as a graph (the so-called {\bf cubic
lattice}), where vectors in $\mathbb{Z}^d$ are the vertices (also
sometimes called {\bf sites}), and two vertices $u,v \in
\mathbb{Z}^d$ are said to be {\bf adjacent}, and we write $u \sim
v$, if $|u - v| = 1$, where $|\cdot|$ is the $L^1$ metric. For
subsets $S,T$ of $\mathbb{Z}^d$, $d(S,T)$ denotes the distance
between $S$ and $T$ using this metric.

We also will use the {\bf lexicographic ordering} on $\mathbb{Z}^d$,
where $v < v'$ if $v \neq v'$ and, for the largest $i$ for which
$v_i \neq v_{i'}$, $v_i$ is strictly smaller than $v_{i'}$. The {\bf
lexicographic past} $\mathcal{P}$ is the set of all $v \in
\mathbb{Z}^d$ smaller than the zero vector.

An {\bf edge} is an unordered pair $(u,v)$ of adjacent sites in
$\mathbb{Z}^d$. The {\bf boundary} of a set $S \subset
\mathbb{Z}^d$, denoted by $\partial S$, is the set of $v \in S^c$
which are adjacent to some element of $S$. In the case where $S$ is
a singleton $\{v\}$, we call the boundary the {\bf set of nearest
neighbors} $N_v$. The {\bf inner boundary} of $S$, denoted by
$\underline{\partial} S$, is the set of sites in $S$ adjacent to
some element of $S^c$, or $\partial(S^c)$.

For any integers $a < b$, we use $[a,b]$ to denote $\{a,a+1,\ldots,b\}$.

An {\bf alphabet} $\A$ is a finite set with at least two elements. A
{\bf configuration} $u$ on the alphabet $\A$ is any mapping from a
non-empty subset $S$ of $\mathbb{Z}^d$ to $\A$, where $S$ is called
the \textbf{shape} of $u$. For any configuration $u$ with shape $S$
and any $T \subseteq S$, denote by $u(T)$ the restriction of $u$ to
$T$, i.e. the subconfiguration of $u$ occupying $T$. For $S,T
\subset \zz^d$, $x \in \A^S$ and $y \in \A^T$, $xy$ denotes the
configuration on $S \cup T$ defined by $(xy)(S) = x$ and $(xy)(T) =
y$, which we call the {\bf concatenation} of $x$ and $y$ (if $S \cap
T \ne \emptyset$, this requires that $x(S \cap T) = y (S \cap T)$).
For a symbol $a \in \A$ and a subset $S \subseteq \zz^d$, $a^S$
denotes the configuration on $S$ which takes value $a$ at all
elements of $S$.

For any $d$, we use $\sigma$ to denote the natural {\bf shift
action} on $\A^{\mathbb{Z}^d}$ defined by $(\sigma_{v}(x))(u) =
x(u+v)$.

For any alphabet $\A$, $\A^{\mathbb{Z}^d}$ is a topological space
when endowed with the product topology (where $\A$ has the discrete
topology), and any subsets will inherit the induced topology. A
basis for the topology is the collection of \textbf{cylinder sets}
which are sets of the form $[w] := \{x \in \A^{\mathbb{Z}^d} \ : \
x(S) = w\}$, where $w$ is a configuration with arbitrary finite
shape $S \subseteq \mathbb{Z}^d$.

A \textbf{subshift} (or \textbf{shift space}) is a closed,
translation invariant subset of $\A^{\zz^d}$.  An equivalent
definition is given as follows. Let $\A^*$ denote the set of all
configurations on finite subsets of $\zz^d$, where we often identify
two configurations if they differ by a translate.  A subset $X$ of
$\A^{\zz^d}$ is a subshift iff
$$
X = \{ x \in \A^{\zz^d}: ~ x(S) \not\in \F ~ \mbox{ for all finite
subsets } S \subset \zz^d\}
$$
for some list $\F \subset \A^*$ of configurations on finite subsets.
For a subshift $X$, when we wish to emphasize the dimension of the
lattice we will refer to $X$ as a \textbf{$\zz^d$-subshift}.

In the case where $\F$ can be chosen to be finite, $X$ is called a
\textbf{shift of finite type} (SFT). In the case where $\F$ consists
of configurations only on edges, $X$ is called a
\textbf{nearest-neighbor shift of finite type}.

The following are prominent examples of nearest neighbor SFT's.

\begin{example}
The \textbf{hard square shift} is the nearest-neighbor
$\mathbb{Z}^d$-SFT with alphabet $\{0,1\}$ defined by forbidding 1's
on any adjacent pair of sites.
\end{example}

\begin{example}
The \textbf{$k$-checkerboard} (or \textbf{$k$-coloring}) SFT
$C^{(d)}_k$ is the nearest-neighbor $\mathbb{Z}^d$-SFT with alphabet
$\{0,1, \ldots, k-1\}$ consisting of all configurations on $\zz^d$
such that letters at adjacent sites must be different.
\end{example}

\begin{definition}
For any $\mathbb{Z}^d$ subshift $X$, the \textbf{language} of $X$ is
$$
\L(X) = \bigcup_{\{S \subset \zz^d, ~ |S| < \infty\}} \L_S(X)
$$
where
$$
\L_S(X) = \{x(S): x \in X\}.
$$
\end{definition}


Given a forbidden list $\F$ that defines a subshift $X$ and $S
\subseteq \zz^d$, every configuration on $S$ that does not contain
any element of $\F$ is called \textbf{locally admissible}; every
configuration on $S$ that extends to an element of $X$ is called
\textbf{globally admissible}. Note that every locally admissible
configuration is globally admissible, but not conversely. Clearly,
a finite configuration is globally admissible if and only if it is in $\L(X)$.

\begin{definition}
A \nn ~SFT $X$ is \textbf{single-site fillable (SSF)}
if for  some forbidden list $\F$ of nearest neighbors that defines
$X$ and every $\eta \in \A^{N_0}$, there exists $a \in \A^{0}$ such
that $\eta a$ is locally admissible.
\end{definition}

It is easy to see that a \nn ~SFT $X$ satisfies SSF if and only if
for some forbidden list $\F$ of nearest neighbors that defines $X$,
every locally admissible configuration is globally admissible.

In the definition of SSF above, the symbol $a$ may depend on the
configuration $\eta$.  This generalizes the concept of a
\textbf{safe symbol}, which is a symbol $a \in \A^{\{0\}}$ such that
$\eta a$ is locally admissible for every configuration $\eta \in
\A^{N_0}$ (strictly speaking, the concept of safe symbol applies to
a forbidden list on the alphabet of symbols that occur in a point of
the subshift). The hard square shift has a safe symbol in every
dimension. No checkerboard shift has a safe symbol, but for $k \ge
2d+1$, $C^{(d)}_k$ satisfies SSF.

\subsection{Markov Random Fields and Gibbs Measures}
\label{MRF}
~

We will frequently speak of measures on $\A^{\zz^d}$, and all such
measures in this paper will be Borel  probability measures. This
means that any $\mu$ is determined by its values on the cylinder
sets.  For notational convenience, rather than referring to a
cylinder set $[w]$ within a measure or conditional measure, we just
use the configuration $w$. For instance, $\mu(w \cap v \ | \ u)$
represents the conditional measure $\mu([w] \cap [v] \ | \ [u])$. By
the \textbf{support}, $\supp(\mu)$, of $\mu$, we mean the
topological support, i.e., the smallest closed set of full measure.
Note that for a configuration $w$ on a finite set, $[w]$ intersects
$\supp(\mu)$ iff $\mu(w) > 0$.

A measure $\mu$  on $\A^{\zz^d}$ is \textbf{shift-invariant} (or
\textbf{stationary} or \textbf{translation-invariant})  if $\mu(A) =
\mu(\sigma_{v} A)$ for all measurable sets $A$ and $v \in \zz^d$.

\begin{definition}
A shift-invariant $\zz^d$-measure $\mu$ is a \textbf{$\zz^d$ Markov
random field} (or MRF) if, for any finite $S \subset \zz^d$, any
$\eta \in \A^S$, any finite $T \subset \zz^d$ s.t. $\partial S
\subseteq T \subseteq \zz^d \setminus S$, and any $\delta \in \A^T$
with $\mu(\delta)
> 0$,
\begin{equation}\label{MRFdefn}
\mu(\eta \ | \ \delta(\partial S)) = \mu(\eta \ | \ \delta).
\end{equation}
\end{definition}

Informally, $\mu$ is an MRF if, for any finite $S \subset \zz^d$,
configurations on the sites in $S$ and configurations on the sites
in $\zz^d \setminus (S \cup
\partial S)$ are $\mu$-conditionally independent given a configuration on the sites
in $\partial S$.

\begin{definition}
For any Markov random field  $\mu$, any finite $S \subseteq \mathbb{Z}^d$, and
any $\delta \in \A^{\partial S}$ with $\mu(\delta) > 0$, define the measure $\mu^{\delta}$ on $\A^S$ by
\[
\mu^{\delta}(w) = \mu(w \ | \ \delta)
\]
for every $w \in \A^S$.
\end{definition}

We will deal mostly with nearest-neighbor Gibbs measures, which are
MRF's specified by nearest-neighbor interactions, defined below.

\begin{definition}
A \textbf{nearest-neighbor interaction} is a shift-invariant
function $\Phi$ from the set of configurations on edges in $\mathbb{Z}^d$
to $\mathbb{R} \cup \infty$. Here, shift-invariance means that
$\Phi(\sigma_v w) = \Phi(w)$ for all configurations $w$ on edges and
all $v \in \zz^d$.
\end{definition}

Clearly, a nearest-neighbor interaction is defined by only finitely
many numbers, namely the values of the interaction on configurations
on edges $\{0,e_i\}, i = 1, \ldots, d$.

For a nearest-neighbor interaction $\Phi$, we define its
\textbf{underlying SFT} as follows:
\[
X_\Phi = \{x \in \A^{\mathbb{Z}^d} \ : \ \Phi(x(\{v,v'\})) \neq
\infty, \textrm{ for all } v \sim v' \mbox{ in } \mathbb{Z}^d\}.
\]
Note that $X_\Phi$ is a nearest-neighbor SFT.

\begin{definition}
For a nearest-neighbor interaction $\Phi$, any finite set $S \subset
\mathbb{Z}^d$, and any $w \in \A^{S}$, the \textbf{energy function}
of $w$ with respect to $\Phi$ is
\[
U^{\Phi}(w) := \sum_{e} \Phi(w(e)),
\]
where the sum ranges over all edges $e$ of $S$. The
\textbf{partition function} of $S$ is
\[
Z^{\Phi}(S) := \sum_{w \in \A^{S}} e^{-U^{\Phi}(w)}.
\]
For $\delta \in \A^{\partial S}$ we define
\[
Z^{\Phi, \delta}(S) := \sum_{w \in \A^{S}} e^{-U^{\Phi}(w\delta)}.
\]

In all definitions, we adopt the convention that $\infty + x =
\infty$ for all $x \in \mathbb{R}$.
\end{definition}

\begin{definition}
For any nearest-neighbor interaction $\Phi$, an MRF
$\mu$ is called a \textbf{Gibbs measure} for $\Phi$ if for any
finite set $S \subset \mathbb{Z}^d$ and $\delta \in \A^{\partial S}$
for which $\mu(\delta)
> 0$, we have $Z^{\Phi, \delta}(S) \ne 0$ and, for
any $w \in \A^S$,
\[
\mu^\delta(w)
 = \frac{e^{-U^{\Phi} (w
\delta)}}{Z^{\Phi, \delta}(S)}.
\]
\end{definition}

Note that for a Gibbs measure $\mu$,  $\supp(\mu)$ is automatically
contained in the underlying SFT $X_\Phi$; we have allowed the
interaction to take on infinite values in order to allow our Gibbs
measures to be supported on proper subsets of $\A^{\zz^d}$. In our
main results, we will assume that $\Phi$ is a nearest-neighbor
interaction and that $X_\Phi$ satisfies a topological
mixing property (the $D$-condition or block $D$-condition described
below) that guarantees $\supp(\mu) = X_\Phi$.

Given a nearest neighbor SFT $X$, and a forbidden list of nearest
neighbor configurations, a \textbf{uniform Gibbs measure} on $X$ is
a Gibbs measure corresponding to the nearest-neighbor interaction
which is 0 on all nearest-neighbor configurations except the
forbidden configurations (on which it is $\infty$).

Every nearest-neighbor interaction $\Phi$ has as least one Gibbs
measure; this is a very special case of a general result of Ruelle
\cite{Ruelle}. Often there are multiple Gibbs measures for a single
$\Phi$; this phenomenon is often called a phase transition. One type
of condition which guarantees uniqueness of Gibbs measures is
so-called spatial mixing, with two variants defined below.

\subsection{Spatial Mixing}
\label{SSM}
~

Let $B_n = [-n,n]^d$, the
d-dimensional cube of side length $2n+1$ centered at the origin.

\begin{definition}
For a function $f(n) : \mathbb{N} \rightarrow \mathbb{R}^+$,
$\lim_{n \rightarrow \infty} f(n) = 0$, we say that an MRF $\mu$
satisfies \textbf{weak spatial mixing \textup{(}WSM\textup{)} with rate $f(n)$} if for
any finite set $S \subseteq B_n$ and any $w \in \A^S$, $\delta,
\delta' \in \A^{\partial B_n}$ s.t. $\mu(\delta), \mu(\delta') > 0$,
\[
 | \mu^{\delta}(w) - \mu^{\delta'}(w) | < |S| f(d(S, \partial
B_n)).
\]
\end{definition}

\begin{definition}
For a function $f(n) : \mathbb{N} \rightarrow \mathbb{R}^+$,
$\lim_{n \rightarrow \infty} f(n) = 0$, we say that an MRF $\mu$
satisfies \textbf{strong spatial mixing \textup{(}SSM\textup{)} with rate $f(n)$} if
for any disjoint finite sets $S,T \subseteq B_n$ and any $v \in
\A^T$, $w \in \A^S$, $\delta, \delta' \in \A^{\partial B_n}$ s.t.
$\mu(\delta), \mu(\delta'), \mu^\delta(v), \mu^{\delta'}(v) > 0$,
\[
| \mu^{\delta}(w \ | \ v) - \mu^{\delta'}(w \ | \ v) | < |S| f(d(S,
\partial B_n)).
\]
\end{definition}

Informally, weak spatial mixing means that conditioning on a
boundary configuration does not have much effect on the measure of a
configuration on a set $S$ far from the boundary, and strong spatial
mixing means that this is still true even if one first conditions on
a configuration on sites which may be close to $S$.

The factor of $|S|$ on the right-hand side is unavoidable in both
definitions; without it weak spatial mixing would force a much more
stringent condition on $\mu$ called $m$-dependence. (see
\cite{bradley})

It is well-known that for a nearest-neighbor interaction $\Phi$,
weak spatial mixing at any rate implies that there is only one Gibbs
measure defined by $\Phi$~\cite[Proposition 2.2]{Weitz2}. Well-known
examples of Gibbs measures that satisfy SSM include the unique
uniform Gibbs measures for the hard square $\zz^2$-SFT and the
$k$-checkerboard $\zz^2$-SFT for $k \ge 12$
(see~\cite{GK},\cite{MP}).

We will need only somewhat weaker spatial mixing conditions for our
main results. These conditions can be formulated in terms of a
function naturally associated to $\mu$ and described as follows.

Recall that $\mathcal{P}$ denotes the lexicographic past in $\zz^d$.
For any shift-invariant measure $\mu$, define the function
\[
p_{\mu}(x) := \mu(x(0) \ | \ x(\mathcal{P}))
\]
which is defined $\mu$-a.e. on $\supp(\mu)$. Note that $p_\mu(x)$
depends only on $x(\PP \cup \{0\})$ for $\mu$-a.e. $x$. Recall that
$I_{\mu}(x) := -\log p_{\mu}(x)$ is known as the information
function of $\mu$.

For any finite $S \subset \zz^d$, define the function
\[
p_{\mu,S}(x) := \mu(x(0) \ | \ x(S)).
\]
We will sometimes refer to the special case
$$
p_{\mu,n}(x) = p_{\mu,\PP_n}(x), \textrm{ where } \PP_n = B_n \cap \PP.
$$

\begin{definition}
We write $\lim_{S \rightarrow \PP} p_{\mu,S}(x) = L$ to mean the
limit in the ``net'' sense: for any $\epsilon
> 0$, there exists $n$ such that for all finite $S$ satisfying
$\PP_n \subset S \subset \PP$, we have $|p_{\mu, S}(x) - L| <
\epsilon$.
\end{definition}

By martingale convergence, $\lim_{n \rightarrow
\infty} p_{\mu,n}(x) = p_\mu(x)$ for $\mu$-a.e. $x \in \supp(\mu)$.
For this reason, for any $x \in \supp(\mu)$, if  $\lim_{n \rightarrow \infty}
p_{\mu,n}(x)$ exists (and hence if $\lim_{S \rightarrow \PP}
p_{\mu,S}(x)$ exists), we will take $p_\mu(x)$ to be this limit.
While we do not know if
$p_{\mu,S}(x)$ always converges a.e. as $S \rightarrow \PP$ in the net sense, $p_{\mu,S}(x)$
is known to converge in $L^1$ as $S \rightarrow \PP$ in the net sense~\cite[Theorem 1.2]{Walsh}.

Our main results, Theorem~\ref{easythm} and~\ref{hardthm}, will
establish a representation for pressure in terms of a given
shift-invariant measure $\nu$.  For Theorem~\ref{easythm} we assume
$\lim_{S \rightarrow \PP} p_{\mu,S}(x)  = p_\mu(x)$ uniformly on
$\supp(\nu)$, i.e. that $\forall \epsilon > 0$ $\exists N > 0$ so that
$\forall x \in \supp(\nu)$,
$\PP_n \subset S \subset \PP \Longrightarrow |p_{\mu, S}(x) - p_\mu(x)| < \epsilon$.

For Theorem~\ref{hardthm}, we will need a stronger type of convergence.

\begin{definition}
We write  $\lim_{S \rightarrow \PP, ~ U \rightarrow +\infty}
p_{\mu,S \cup U}(x) = L$ to mean that  for any $\epsilon
> 0$, there exists $n$ such that for all finite $S,U$ satisfying
$\PP_n \subset S \subset \PP$ and  $U \subset (B_n \cup \PP)^c$, we
have $|p_{\mu, S \cup U}(x) - L| < \epsilon$.
\end{definition}

We note that this definition could also be written
in terms of a single set; the only property required of $S \cup U$ is
that $S \cup U$ contains $\PP_n$ and is contained in $B_n^c \cup \PP_n$.
The definition is written with $S$ and $U$ decoupled only to make
comparisons to $\lim_{S \rightarrow \PP} p_{\mu,S}$ more clear.

Again, we will take $p_\mu(x)$ to be the value of this limit when it
exists. For Theorem~\ref{hardthm}, we will assume $\lim_{S
\rightarrow \PP, ~ U \rightarrow +\infty} p_{\mu,S \cup U}(x) =
p_\mu(x)$ uniformly on $\supp(\nu)$.  Clearly this implies $\lim_{S
\rightarrow \PP} p_{\mu,S}(x) = p_\mu(x)$ uniformly on $\supp(\nu)$.

We have the following implication.

\begin{proposition}
\label{SSM_continuity} For an MRF $\mu$, if $\mu$ satisfies SSM at
any rate, then  \newline $\lim_{S \rightarrow \PP, U \rightarrow
+\infty} p_{\mu,S \cup U}(x) = p_\mu(x)$ uniformly on $\supp(\mu)$.
\end{proposition}

\begin{proof}

We find it convenient to define the vector-valued function
\[
\hat{p}^n_{\mu}(x) := \mu(y(0) = \cdot \ | \ y(\partial S_n) =
x(\partial S_n))
\]
where $S_n = B_n \setminus \PP_n$ (so, $\hat{p}^n_{\mu}(x)_a =
\mu(y(0) = a \ | \ y(\partial S_n) = x(\partial S_n))$).

By SSM applied to $S = \{0\}$ and $T_n = \PP_n = \mathcal{P} \cap
B_n$. we see that given $\epsilon
> 0$,  for $n$ sufficiently large, if $x, x' \in
\supp(\mu)$, and $x(T_n) = x'(T_n)$, then $|\hat{p}^n_{\mu}(x) -
\hat{p}^n_{\mu}(x')| < \epsilon$. For $m \ge n$,
$\hat{p}^m_{\mu}(x)$ can be written as a weighted average of
$\hat{p}^n_{\mu}(x')$ for finitely many $x'$ agreeing with $x$ on
$T_n$. Thus, $|\hat{p}^n_{\mu}(x) - \hat{p}^m_{\mu}(x)| < \epsilon$.
So, the sequence $\hat{p}^n_{\mu}$ is uniformly Cauchy and therefore
uniformly convergent.

We can decompose $\partial S_n$ as a disjoint union
$$
\partial S_n = U_n \cup C_n
$$
where  $U_n = (\partial S_n) \cap \mathcal{P}$ 
(the ``upper layer'' of $\PP_n$) and
$C_n = \partial S_n \setminus U_n$ (a``canopy'' sitting over $U_n$).

Fix $n$ and let $S$ be a finite set satisfying $\PP_n \subset S
\subset \PP$ and $U \subset (B_n \cup \PP)^c$. Then

\begin{multline}
\label{eqn_wtd2} p_{\mu,S \cup U}(x) = \sum_{\delta \in \A^{C_n}: ~
\mu(x(S \cup U)\delta) > 0} ~ \mu(x(0) \ | \ x(S
\cup U), \delta) \mu(\delta \ | \ x(S \cup U)) \\
 = \sum_{\delta \in
\A^{C_n}: ~ \mu(x(S \cup U) \delta) > 0} ~ \mu(x(0) \ | \ x(U_n),
\delta) \mu(\delta \ | \ x(S \cup U)) \\
\\
 = \sum_{\delta \in \A^{C_n}: ~ \mu(x(S \cup U)
\delta) > 0} ~ \mu(y(0) = x(0) ~ | ~ y(\partial S_n) =
y_\delta(\partial S_n)) \mu(\delta \ | \ x(S \cup U)),
\end{multline}
where $y_\delta$ is any point in $\supp(\mu)$  such that
$y_\delta(U_n) = x(U_n)$ and $y_\delta(C_n) = \delta$.
(For instance, any $y \in [x(S \cup U) \delta]$.)

Let $g(x)$ denote the (vector-valued) uniform limit of
$\hat{p}^n_\mu$ on $\supp(\mu)$. It follows from the above that
given $\epsilon
> 0$, for sufficiently large $n$,
$$
|\mu(y(0) = x(0) ~ | ~ y(\partial S_n)= y_\delta(\partial S_n)) -
g(x)_{x(0)} | < \epsilon.
$$
Thus, $\lim_{S \rightarrow \PP, U \rightarrow +\infty} p_{\mu,S \cup
U}(x) = g(x)_{x(0)}$ uniformly on $\supp(\mu)$ and $g(x)_{x(0)} =
p_\mu(x)$ by our convention.
\end{proof}

\subsection{Positivity of $p_\mu$}
\label{positivity}
~

Positivity of $p_\mu$ and related functions will play an important
role in our main results.  We begin with an easy implication between
two forms of positivity.

\begin{definition}
For shift-invariant measures $\mu, \nu$, with $\supp(\nu) \subseteq \supp(\mu)$, define
$$
c_{\mu,\nu} := \inf_{x \in \supp(\nu), ~ S \subset \PP, ~ |S| < \infty} ~
p_{\mu,S}(x).
$$
and
$$
c_\mu := c_{\mu,\mu}.
$$
\end{definition}

\begin{proposition}
\label{pos}
If $c_\mu > 0$, then $p_\mu$ is bounded away from zero $\mu$-a.e.,~
i.e., the essential infimum of $p_\mu$ is positive.
\end{proposition}

\begin{proof}
If $c_\mu > 0$, then the functions $p_{\mu,n}$ are uniformly bounded
away from zero on $\supp(\mu)$.  Since  $p_{\mu,n}$ converges to
$p_\mu$ $\mu$-a.e., it follows that $p_\mu$ is bounded away from
zero $\mu$-a.e.
\end{proof}

Next, we show that SSF is sufficient for a stronger form of
positivity.

\begin{proposition}\label{SSFpos}
\label{SSF_positive} If $\Phi$ is a nearest-neighbor interaction and
$X_\Phi$ satisfies SSF, then for any Gibbs measure  $\mu$ for
$\Phi$, $c_{\mu} > 0$.
\end{proposition}

\begin{proof}

We will show in fact that
$$
\inf_{x \in \supp(\mu), ~ S \subset \zz^d \setminus \{0\}, ~ |S| <
\infty} ~ p_{\mu, S}(x) > 0.
$$

Recall that $N_0$ denotes the set of nearest neighbors of 0. Let $S
\subset \zz^d \setminus \{0\}$ be a finite set. Let
\begin{itemize}
\item
$U = N_0 \setminus S$
\item
$V = (\partial (\{0\} \cup U)) \setminus S$
\item
$S' = (\partial (\{0\} \cup U)) \cap S$
\end{itemize}

In particular, $\partial (\{0\} \cup U)$ is the disjoint union of
$V$ and $S'$.

Let $x \in X$. There exists $v \in \A^V$ such that
$$
\mu(v ~ | ~ x(S)) \ge |\A|^{-|V|}.
$$
Let $L$ and $\ell$ be upper and lower bounds on finite values of
$\Phi$. Since any locally admissible configuration is globally
admissible, there exists $u \in \A^U$ such that $x(0)x(S)uv \in
\L(X)$. Thus,
\begin{multline*}
p_{\mu,S}(x) \geq
\mu\left(y(0) = x(0), y(U) = u \ | \ y(S) = x(S)\right) \geq\\
\mu\left(y(0) = x(0), y(U) = u \ | \ y(S) = x(S), y(V) = v\right)
\mu(y(V) = v ~ | ~ y(S) = x(S)) = \\
\mu\left(y(0) = x(0), y(U) = u \ | \ y(S') = x(S'),
y(V) = v\right) \mu(y(V) = v ~ | ~ y(S) = x(S)) \ge  \\
|\A|^{-|U| - 1}e^{(4d^2)(\ell - L)} |\A|^{-|V|} \ge e^{(4d^2)(\ell -
L)} |\A|^{-|N_0| - |N_1| - 1}
\end{multline*}
where $N_1$ is the set of nearest neighbors of $N_0$, other than 0.

Since this lower bound is positive and independent of $x$ and $S$,
we are done.
\end{proof}

The preceding result applies to $X =  C^{(2)}_k$ for any $k \ge 5$.
In contrast we have the following result, for which we need the
following configuration.
Let $x^* \in \{0,1,2\}^{\zz^2}$ be defined
by:
\[
x^*(v) = \begin{cases} v_1 + v_2 \pmod 3  \textrm{ if } v \in
\mathcal{P}
\\ 1 + v_1 + v_2 \pmod 3 \textrm{ if } v \notin \PP
\end{cases}.
\]
We claim that that $x^* \in C^{(2)}_3$. To see this, we must show that
for all $v$ and  $i = 1, 2$, $x^*(v + e_i) \ne x^*(v)$. Clearly this
holds when $v, v+e_i \in \mathcal{P}$ or $v, v+e_i \notin \mathcal{P}$ by the
individual piecewise formulas. When $v \in \mathcal{P}$ and $v + e_i
\not\in \mathcal{P}$, we have $x^*(v + e_i) = 2 + x^*(v)
\ne x^*(v) \pmod 3$.

\begin{proposition}
Let $X =  C^{(2)}_3$ and $\mu$ be a shift-invariant measure with
$\supp(\mu) = X$. Then $c_{\mu} = 0$.
In fact, for any shift-invariant measure $\nu$ such that
$x^* \in \supp(\nu) \subseteq \supp(\mu)$, we have $c_{\mu,\nu} = 0$.
\end{proposition}

\begin{proof}

Fix any $n$, and let $W_{3n} = [1,3n]\times[-3n, -1]$.
For $1 \le i \le n-1$, define $x^i = \sigma_{(-3i,0)} x^*$ and note that
$$
x^i(W_{3n}) = x^*(W_{3n}), ~x^i(3i-1,0) = 2, \textrm{ and } x^i(3i,0) = 1.
$$

We now show that the sets $[x^i(W_{3n} \cup \{(3i-1,0),
(3i,0)\})]$, for $1 \le i \le n-1$, are disjoint.
Choose $1 \leq i < i' \leq n-1$. Let $x \in [x^i(W_{3n} \cup
\{(3i-1,0),(3i,0)\})]$. We claim that for $3i \le j \le 3i'$,
$x(j,0) = j+1 \pmod 3$. To see this, we argue by induction. For $j =
3i$, this is true by definition of $x^i$. Assume this is true for a
given $j$. Then $x(j+1,0) \ne x(j,0) = j + 1 \pmod 3$ and $x(j+1,0)
\ne x(j+1,-1) = x^*(j+1,-1) = j \pmod 3$. Thus, $x(j+1,0) = j+2
\pmod 3$, as desired, completing our proof by induction. But then
$x(3i',0) = 2$, and since $x^{i'}(3i',0) = 1$, $x \notin
[x^{i'}(W_{3n} \cup \{(3i'-1,0),(3i',0)\})]$, and
these cylinder sets are disjoint as claimed.

We now decompose $\mu(x^i(\{(3i-1,0), (3i,0)\}) ~|~ x^*(W_{3n}))$ as
$$
\mu(x^i(3i-1,0)~|~ x^*(W_{3n}))
\mu(x^i(3i,0) ~|~ x^*(W_{3n}),  x^i(3i-1,0)).
$$
Since $W_{3n} - (3i-1,0) \subseteq \PP_{3n}$ and $x^*(W_{3n}) = x^i(W_{3n})$,
the first factor can be rewritten as $\displaystyle p_{\mu,W_{3n} - (3i-1,0)}(\sigma_{(3i-1,0)} x^i)$.
Similarly, the second factor can be expressed as $\displaystyle p_{\mu,\{(-1,0)\} \cup (W_{3n} - (3i,0))}(\sigma_{(3i,0)} x^i)$.
Both are greater than or equal to $c_{\mu,\nu}$ by definition, so
$$
\mu(x^i(\{(3i-1,0), (3i,0)\}) ~|~ x^*(W_{3n})) \ge c_{\mu,\nu}^2.
$$
By disjointness of $\{[x^i(W_{3n} \cup \{(3i-1,0), (3i,0))\}]\}$, we obtain
$c_{\mu,\nu}^2(n-1) \le 1$. Since this is true for all $n$, $c_{\mu,\nu} = 0$.
\end{proof}

We remark that there is a fully-supported
nearest-neighbor uniform Gibbs measure  $\mu$ on $C^{(2)}_3$ (\cite{tom-nishant}),
and so the proposition applies to such a $\mu$ and $\nu=\mu$.


\subsection{D-condition}
\label{D-condition}
~

We will frequently make use of a topological mixing condition,
defined by Ruelle~\cite[Section 4.1]{Ruelle} and given in the
definition below. For this, we use the following notation. Let
$\Lambda_n$ be a sequence of finite sets. We write $\Lambda_n
\nearrow \infty$ if $\cup_n \Lambda_n = \zz^d$ and for each $v \in
\zz^d$,
$$
\lim_{n \rightarrow \infty} \frac{|\Lambda_n ~\Delta~ (\Lambda_n +
v)|}{|\Lambda_n|} = 0
$$
where $\Delta$ denotes symmetric difference.

\begin{definition}
An SFT $X$ satisfies the \textbf{D-condition} if there exist
sequences of finite subsets $(\Lambda_n)$, $(M_n)$ of $\mathbb{Z}^d$
such that $\Lambda_n \nearrow \infty$,  $\Lambda_n \subseteq M_n$,
$\frac{|M_n|}{|\Lambda_n|} \rightarrow 1$, and, for any $v \in
\L_{\Lambda_n}(X)$ and finite $S \subset M_n^c$ and $w \in \L_S(X)$,
$[v] \cap [w] \neq \varnothing$.
\end{definition}

Roughly speaking, this means that there exists an exhaustive sequence
of shapes, $\Lambda_n$, for which relatively few elements of
$\Lambda_n$ are near $\partial \Lambda_n$, and ``collars,''
($M_n \setminus \Lambda_n$),
around
the shapes  of comparatively small size
such that it is possible to ``fill in'' between any legal
configurations inside and outside the collar.

We will make use of the following property of a sequence $\Lambda_n
\nearrow \infty$.

\begin{lemma}
\label{vanHove} Let $\Lambda_n \nearrow \infty$. Given $n, N \in
\mN$, let
$$
\Lambda_n^N = \{v \in \Lambda_n: ~ v + \PP_N  \subset \Lambda_n\}.
$$
Then for fixed $N$,
$$
\lim_{n \rightarrow \infty} \frac{|\Lambda_n^N|}{|\Lambda_n|}  = 1.
$$
\end{lemma}

\begin{proof}
If $v \in \Lambda_n$ and $v + \PP_N \not\subset \Lambda_n$, then for
some $w \in \PP_N$, $v \in \Lambda_n \setminus (\Lambda_n - w)$.  Thus,
$$
\frac{|\Lambda_n \setminus \Lambda_n^N|}{|\Lambda_n|}  \le
\frac{\sum_{w \in \PP_N}~ |\Lambda_n \setminus (\Lambda_n - w)|}
{|\Lambda_n|}.
$$
But the right hand side tends to 0 as $n \rightarrow \infty$.
\end{proof}

For most known examples, one can choose $\Lambda_n$ to be
rectangular prisms, or cubes. This motivates the following
variation:

\begin{definition}
An SFT $X$ satisfies the \textbf{block D-condition} if there exists
a sequence of integers $(R_n)$ such that $\frac{R_n}{n} \rightarrow
0$ and for any rectangular prism $B = \prod_{i=1}^d [-n_i, n_i]$,
any integers $m_i \geq R_{n_i}$, and any $w \in \L_{\partial
\left(\prod_{i=1}^d [-(n_i + m_i), n_i + m_i]\right)}(X)$ and $v \in
\L_B(X)$, $[v] \cap [w] \neq \varnothing$.
\end{definition}

For a nearest-neighbor SFT, the block D-condition can be expressed in
the following equivalent
form.

\begin{lemma}
\label{blockD}
A nearest-neighbor SFT $X$ satisfies the block
D-condition if and only if there exists a sequence of integers
$(R_n)$ such that $\frac{R_n}{n} \rightarrow 0$ and for any
rectangular prism $B = \prod_{i=1}^d [-n_i, n_i]$, any integers $m_i
\geq R_{n_i}$, any finite set $S \subset \left(\prod_{i=1}^d [-(n_i
+ m_i), n_i + m_i]\right)^c$, if $w \in \L_S(X)$ and $v \in
\L_B(X)$, then $[v] \cap [w] \neq \varnothing$.
\end{lemma}

\begin{proof}
Let $T =\prod_{i=1}^d [-(n_i + m_i), n_i + m_i]$.  If $w\in \L_S(X)$
for such a set $S$, then we can extend $w$ to a globally admissible
configuration $w'$ on $S \cup
\partial T$. Let $w'' = w'(\partial T)$.  By the block D-condition
$[v]\cap[w''] \ne \varnothing$. Since $X$ is a nearest-neighbor SFT,
it follows that $[v]\cap[w'] \ne \varnothing$ and thus $[v]\cap[w]
\ne \varnothing$.
\end{proof}

It follows, by choosing $\Lambda_n = \prod_{i=1}^d [-n,n]$ and $M_n
= \prod_{i=1}^d [-(n + R_n), n + R_n]$, that for a
nearest-neighbor SFT, the block D-condition
does indeed imply the D-condition, and this implication can be easily
generalized to any SFT. To distinguish between these
definitions we sometimes refer to the D-condition as the classical
D-condition.

For Theorem~\ref{easythm}, we will
assume the classical D-condition. For Theorem~\ref{hardthm}, we will
need the block D-condition. However, we are not aware of any example
which satisfies the classical D-condition and not the block
D-condition; in fact in some works (e.g., \cite{tom}) the
D-condition is stated with the assumption that the sets $\Lambda_n$
and $M_n$ are cubes or rectangular blocks.

We will make use of  the following result that is well known under a
much weaker hypothesis than the D-condition
(\cite[Remark 1.14]{Ruelle}). We give a proof for completeness.

\begin{proposition}
\label{full}  If $\Phi$ is a nearest-neighbor interaction and
$X_\Phi$ satisfies the $D$-condition, then for any Gibbs measure
$\mu$ for $\Phi$, $\supp(\mu) = X_\Phi$.
\end{proposition}

\begin{proof}
As mentioned earlier, $\supp(\mu) \subseteq X_\Phi$.  Let $X =
X_\Phi$.

Let $S$ be a finite subset of $\zz^d$ and $w \in \L_S(X)$.  By the
$D$-condition there exists a finite set $T$ containing $S$ such that
for any $\delta \in \L_{\partial T}(X)$, we have $w\delta \in
\L(X)$. Since there exists some $\delta \in \A^{\partial T}$ such
that $\mu(\delta) > 0$ and since  $\supp(\mu) \subseteq X$, we have
$\delta \in \L_{\partial T}(X)$,  Thus, $w\delta \in \L(X)$ and so
$\mu(w|\delta) > 0$. Since
$$
\mu(w) = \sum_{\eta \in \A^{\partial T}:~ \mu(\eta) > 0} ~ \mu(w
| \eta)\mu(\eta)
$$
we have $\mu(w) > 0$.
\end{proof}

\begin{proposition}
\label{SSF_D}
Any $\zz^d$ SFT that satisfies SSF must satisfy the
block D-condition.
\end{proposition}

\begin{proof}
It is clear that the block D-condition is satisfied with each $m_i = 1$,
since for any dimensions $n_i$, $1 \leq i \leq d$, and configurations
$w \in \L_{\partial{\prod_{i=1}^d [-(n_i + 1), n_i + 1]}}(X)$ and
$v \in \L_{\prod_{i=1}^d [-n_i, n_i]}(X)$, the concatenation $vw$ is locally
admissible, therefore globally admissible by SSF, and so $[v] \cap [w] \neq \varnothing$.
\end{proof}

Recall that $d = 2$ and $k \ge 5$, the $k$-checkerboard SFT satisfies SSF
and therefore satisfies the block-D condition. On the other hand,
the 3-checkerboard SFT does not even satisfy the
classical D-condition.
This follows from the existence of so-called frozen points,
as shown in~\cite[p. 253]{Schmidt1}.  We give a proof for completeness.

\begin{proposition}\label{noD}
The $3$-checkerboard SFT does not satisfy the classical D-condition for
any $d \ge 2$.
\end{proposition}

\begin{proof}
For any $d$, let $x_d$ be defined by $x_d(v) = \left(\sum_i~ v_i\right)
\pmod 3$ for all $v \in \mathbb{Z}^d$.  Then $x_d \in C^{(d)}_3$.

We will show that there are no
points $x \in C^{(d)}_3$, $x \ne x_d$, which agree with $x_d$
on all but finitely many sites.
(Such points are sometimes called {\bf frozen}.)

We argue for $d = 2$, which implies the same for all $d$
since the restriction of $x_d$ to any translate of
$\{v \in \zz^d: v_i = 0 \mbox{ for all } i > 2\}$
agrees with a shift of $x_2$

Suppose that $y \in C^{(2)}_3$ agrees with $x_2$ on the
complement of a finite set of sites.  Let $S$ be the set
of sites at which $y$ and $x_2$ disagree.
Consider the leftmost site
$v$ in the top row of $S$. Its neighbors $v - e_1$
and $v + e_2$ are in $S^c$,
and by definition of $x_2$,
$y(v - e_1) = x_2(v - e_1) \neq x_2(v + e_2) = y(v + e_2)$.
Therefore, there is only one legal choice for $y(v)$, which is $x_2(v)$,
contrary to the fact that $v \in S$.

In particular, this implies that for any $d$ and
finite $S \subseteq \mathbb{Z}^d$, any boundary configuration
$x_d(\partial S)$ has only one valid completion to all of $S \cup \partial S$,
which precludes the classical D-condition (for instance, one cannot ``fill in'' between
the configuration on any cube
which consists of alternating 0's and 1's and the restriction of
$z_d$ to the boundary of any larger cube).
\end{proof}

\subsection{Entropy, Pressure and Equilibrium States}
\label{pressure}
~

For a shift-invariant
measure $\mu$ on $\A^{\mathbb{Z}^d}$, we define its entropy as
follows.

\begin{definition}
The \textbf{measure-theoretic entropy} of a shift-invariant measure
$\mu$ on $\A^{\mathbb{Z}^d}$ is
defined by
\[
h(\mu)=\lim_{j_1, j_2, \ldots, j_d \rightarrow \infty} \frac{-1}{j_1
j_2 \cdots j_d} \sum_{w \in \A^{\prod_{i=1}^d [1,j_i]}} \mu(w)
\log(\mu(w)),
\]

\noindent where terms with $\mu(w) = 0$ are omitted.
\end{definition}

We define topological pressure for both interactions and functions
on a shift space $X$.  In order to discuss connections between these
viewpoints, we need a mechanism for turning an interaction (which is
a function on finite configurations) into a continuous function on
the infinite configurations in $X$. Following Ruelle, we do this as
follows for the special case of nearest-neighbor interactions
$\Phi$. Define for $x \in X_\Phi$
$$
A_{\Phi}(x) := -\sum_{i=1}^d \Phi(\{x_0, x_{e_i}\}).
$$

We now give the two definitions of topological pressure.

\begin{definition}
For a nearest-neighbor interaction $\Phi$
the \textbf{(topological) pressure} of $\Phi$ is defined as
\[
P(\Phi) = \lim_{n_1, \ldots, n_d \rightarrow \infty} \frac{1}{\prod n_i} \log Z^{\Phi}\left(\prod [1,n_i]\right).
\]
\end{definition}

It is well-known~\cite[Corollary 3.13]{Ruelle} that for any sequence
$\Lambda_n \nearrow \infty$,
\[
P(\Phi) = \lim_{n \rightarrow \infty} \frac{1}{|\Lambda_n|} \log
Z^{\Phi}\left(\Lambda_n\right).
\]

\begin{definition}
For any continuous real-valued function $f$
on a $\mathbb{Z}^d$ SFT $X$, the \textbf{(topological) pressure} of $f$ is defined as
\[
P(f) = \sup_{\mu}~ h(\mu) + \int f \ d \mu,
\]
where the supremum ranges over  all shift-invariant measures $\mu$
supported on $X$. Any $\mu$ achieving this supremum is called an
\textbf{equilibrium state} for $f$.

When $f = 0$, $P(f)$ is called the \textbf{topological entropy} $h(X)$ of $X$,
and any equilibrium state for $f$ is called a \textbf{measure of maximal entropy} for $X$.
\end{definition}

It is well-known
that for an irreducible nearest-neighbor $\zz$-SFT $X$, there is a
unique uniform Gibbs measure; this measure is the unique measure of
maximal entropy on $X$~\cite[Section 13.3]{LM} (which is an
irreducible  (first-order) Markov chain).

The celebrated Variational Principle
\cite{walters_var}~\cite{walters} implies that the definitions we
have given are equivalent in the sense that $P(\Phi) = P(A_{\Phi})$.
It is well-known that any continuous $f$ has at least one
equilibrium state~\cite{walters}.  And in the case that $X_\Phi$
satisfies the $D$-condition, a measure on $X_\Phi$ is a Gibbs
measure for $\Phi$ iff it is an equilibrium state
for $A_\Phi$~\cite{dob},~\cite{LR},~\cite[Theorem 4.2]{Ruelle}. We
will discuss connections, in a more general context, between Gibbs
states and equilibrium states in Section~\ref{connections}.

\section{Main results}\label{main}

\begin{theorem}\label{easythm}
If $\Phi$ is a nearest-neighbor interaction with underlying SFT $X
= X_\Phi$, $\mu$ is a Gibbs measure for $\Phi$, $\nu$ is a
shift-invariant measure with
$\supp(\nu) \subseteq X$, \\

\noindent {\rm (A1)} $X$ satisfies the classical D-condition,

\noindent {\rm (A2)} $\lim_{S \rightarrow \PP} ~p_{\mu, S}(x)=
p_\mu(x)$ uniformly over $x \in \supp(\nu)$, and

\noindent {\rm (A3)} $
c_{\mu,\nu} > 0$

%


\noindent then
$$
P(\Phi) = \int I_{\mu}(x) + A_{\Phi}(x) \ d\nu =
\int I_{\mu}(x) - \sum_{i=1}^d \Phi(x(\{0,e_i\})) \ d\nu.
$$

\end{theorem}

Theorem 1 of~\cite{GK}, which motivated our paper, shows that if
$\mu$ is a Gibbs measure for a nearest-neighbor interaction $\Phi$
and $X_\Phi$ has a safe symbol $a$ and satisfies SSM, then the
pressure representation above holds for the point mass $\nu$ on
$a^{\zz^d}$. We remark that Theorem~\ref{easythm} generalizes this
result, with weaker hypotheses and a stronger conclusion. To see
this, first recall that the existence of a safe symbol is even
stronger than SSF, which implies (A1) and (A3) by
Propositions~\ref{SSF_D} and~\ref{SSF_positive}; second, recall from
Proposition~\ref{SSM_continuity} that SSM implies (A2).

\begin{proof}

Recall from Proposition~\ref{full} that $\supp(\mu) = X$. So,
$p_\mu$ is defined $\mu$-a.e. on $X$ and for any finite set $S$,
$w \in \L_S(X)$ iff $\mu(w) > 0$.

Choose $\ell < 0$ and $L > 0$ to be lower and
upper bounds respectively on finite values of $\Phi$.
Let $\Lambda_n, M_n$ be as in the definition of the
D-condition.

We begin by proving that
\begin{equation}\label{subexperror}
\frac{1}{|\Lambda_n|} \left( \log Z^{\Phi}(\Lambda_n) + \log \mu(x(\Lambda_n))
+ U^{\Phi}(x(\Lambda_n)) \right) \rightarrow 0
\end{equation}
uniformly in $x \in X$
(though we will need this only for $x \in \supp(\nu)$).
For this, we will only use the D-condition (A1).

Fix $n$ and let $R_n = |M_n| - |\Lambda_n|$.
Note that for any $w \in \L_{\Lambda_n}(X)$,
\begin{equation}
\label{average} \mu(w) = \sum_{\delta \in \L_{\partial M_n}(X)} ~
\mu(w ~|~ \delta) \mu(\delta).
\end{equation}
For any such $w$ and $\delta$, by the D-condition
there exists $y_{w,\delta} \in \L_{M_n \setminus \Lambda_n}(X)$
such that $wy_{w,\delta}\delta \in \L(X)$.  Then there is a constant
$C_d > 0$ such that

\begin{multline*}
\mu(w \ | \ \delta) \geq \mu(wy_{w,\delta} \ | \ \delta) =
\frac{e^{-U^{\Phi}(wy_{w,\delta}\delta)}}
{\sum_{u \in \L_{M_n}(X)} e^{-U^{\Phi}(u\delta)}} \\
\geq \frac{e^{-U^{\Phi}(w) - C_d R_n L}}{\sum_{v \in
\L_{\Lambda_n}(X)} e^{-U^{\Phi}(v)} |\A|^{C_dR_n} e^{- C_dR_n \ell}}
= \frac{e^{-U^{\Phi}(w)}}{Z^\Phi(\Lambda_n)} e^{R_n (C_d\ell - C_dL
- C_d\log |\A|)}.
\end{multline*}

Let $y_{\max}$ achieve $\max  \mu(wy \ | \ \delta)$ over all
$y \in \L_{M_n \setminus \Lambda_n}(X)$.
Then,
\begin{multline*}
\mu(w \ | \ \delta) = \sum_{y \in \L_{M_n \setminus \Lambda_n}(X)}
\mu(wy \ | \ \delta) \leq |\A|^{R_n} \mu(wy_{\max} \ | \ \delta)
\\ = |\A|^{R_n}
 \frac{e^{-U^{\Phi}(wy_{\max} \delta)}}{\sum_{u \in \L_{M_n}(X)}
 e^{-U^{\Phi}(u\delta)}}
\leq |\A|^{R_n} \frac{e^{-U^{\Phi}(w) - C_dR_n \ell}}{\sum_{v
\in \L_{\Lambda_n}(X)}
e^{-U^{\Phi}(v) - C_d R_n L}} \\
= \frac{e^{-U^{\Phi}(w)}}{Z^\Phi(\Lambda_n)} e^{R_n (C_d L -
C_d\ell + \log |\A|)}.
\end{multline*}

Since $\sum_{\delta} \mu(\delta) = 1$,
we can combine the three formulas above to see that
\[
\gamma^{-R_n} \leq \mu(w) Z^\Phi(\Lambda_n)
e^{U^{\Phi}(w)} \leq \gamma^{R_n},
\]
where $\gamma := e^{C_d\log |\A| + C_dL - C_d\ell} > 0$. Since
$\frac{R_n}{|\Lambda_n|} \rightarrow 0$, this implies
(\ref{subexperror}).

We use (\ref{subexperror}) to represent pressure:
\[
P(\Phi) = \lim_{n \rightarrow \infty}
\frac{\log Z^\Phi(\Lambda_n)}{|\Lambda_n|}
= \lim_{n \rightarrow \infty} \int \frac{\log Z^\Phi(\Lambda_n)}{|\Lambda_n|} \ d\nu
= \lim_{n \rightarrow \infty} \int \frac{-U^{\Phi}(x(\Lambda_n))
- \log \mu(x(\Lambda_n))}{|\Lambda_n|} \ d\nu.
\]
(Here the second equality comes from the fact that
$\frac{\log Z^\Phi(\Lambda_n)}{|\Lambda_n|}$ is independent of $x$,
and the third from (\ref{subexperror}).)
Since $\nu$ is shift-invariant and $\Lambda_n \nearrow \infty$,
$$
\lim_{n \rightarrow \infty} \int
\frac{-U^{\Phi}(x(\Lambda_n))}{|\Lambda_n|} \ d\nu = \int A_{\Phi}(x) \ d\nu,
$$
and so we can write
\[
P(\Phi) = \int A_{\Phi}(x) \ d\nu
- \lim_{n \rightarrow \infty}
\int \frac{\log \mu(x(\Lambda_n))}{|\Lambda_n|} \ d\nu.
\]

It remains to show that
$\lim_{n \rightarrow \infty} \int
\frac{- \log \mu(x(\Lambda_n))}{|\Lambda_n|} \ d\nu = \int I_{\mu}(x) \ d\nu$.
We do this by decomposing $\mu(x(\Lambda_n))$ as a product of
conditional probabilities. Denote by
 $(s^{(n)}_i)_{i=1}^{|\Lambda_n|}$ the sites of
 $\Lambda_n$, ordered lexicographically.
 For any $1 \leq i \leq |\Lambda_n|$,
 denote by $S^{(n)}_i$ the set
 $\{s^{(n)}_j: ~ 1 \le j \le i-1\}$. (This means that
 $S^{(n)}_1 = \varnothing$). Then for any $x \in \supp(\nu)$, we can write
\begin{equation}\label{decomp}
- \log \mu(x(\Lambda_n)) = \sum_{i = 1}^{|\Lambda_n|}
- \log \mu \left(x(s^{(n)}_i) \ | \ x(S^{(n)}_i)\right)
= \sum_{i=1}^{|\Lambda_n|} - \log p_{\mu, S^{(n)}_i - s^{(n)}_i} (\sigma_{s^{(n)}_i} x).
\end{equation}

Clearly, each term $- \log p_{\mu, S^{(n)}_i - s^{(n)}_i} (\sigma_{s^{(n)}_i} x)$ is
lower bounded by $0$. To get an upper bound, we use shift-invariance
of $\mu$, the fact that $S^{(n)}_i - s^{(n)}_i \subset \PP$ and
Assumption (A3) to conclude that for any $x \in \supp(\nu)$ and all $n,i$,
$$
p_{\mu, S^{(n)}_i - s^{(n)}_i} (\sigma_{s^{(n)}_i} x) \ge c,
$$
where $c:= c_{\mu,\nu}$.

Fix $\epsilon > 0$ and $N \in \mN$.

By Lemma~\ref{vanHove}, for sufficiently large $n$,
$$
\frac{|\Lambda_n^N|}{|\Lambda_n|}  > 1 -
\epsilon.
$$
(recall that $ \Lambda_n^N = \{v \in \Lambda_n: ~ v + \PP_N  \subset
\Lambda_n\}. $)

For all $i$ such that $s^{(n)}_i \in \Lambda_n^N$, we have $\PP_N
\subset S^{(n)}_i - s^{(n)}_i \subset \PP$.  By assumption (A2), we
then have that for sufficiently large $N$, all $i$ such that
$s^{(n)}_i \in \Lambda_n^N$ and all $x \in \supp(\nu)$,
$$
\left|  p_{\mu, S^{(n)}_i -
s^{(n)}_i} (\sigma_{s^{(n)}_i} x) - p_{\mu}(\sigma_{s^{(n)}_i}
x)\right| < \epsilon.
$$

Since $p_{\mu, S^{(n)}_i - s^{(n)}_i}$ is bounded from below
by $c$ and $p_{\mu,n}$ converges to $p_\mu$ on $\supp(\nu)$, it
follows that $p_\mu$ is also bounded below by
c
on $\supp(\nu)$.
Thus,
\[
\left|- \log p_{\mu, S^{(n)}_i - s^{(n)}_i} (\sigma_{s^{(n)}_i} x) -
I_{\mu}(\sigma_{s^{(n)}_i} x)\right| < \epsilon/c
\]

Therefore,
\[
\left|\int- \log p_{\mu, S^{(n)}_i - s^{(n)}_i} (\sigma_{s^{(n)}_i} x) \
d\nu - \int I_{\mu}(\sigma_{s^{(n)}_i} x) \ d\nu\right| <
\epsilon/c.
\]

Since $\nu$ is shift-invariant,
we have
\[
\left|\int - \log p_{\mu, S^{(n)}_i - s^{(n)}_i} (\sigma_{s^{(n)}_i}
x) \ d\nu - \int I_{\mu}(x) \ d\nu\right| < \epsilon/c.
\]

Now, we are prepared to give bounds on (\ref{decomp}). By the preceding,
$$
\left|\sum_{s^{(n)}_i \in \Lambda_n^N} \int - \log p_{\mu, S^{(n)}_i - s^{(n)}_i} (\sigma_{s^{(n)}_i} x) \ d\nu - |\Lambda_n^N| \int I_{\mu}(x) \
d\nu \right| \le |\Lambda_n| (\epsilon/c).
$$
Also, since
$0 \le \int - \log p_{\mu, S^{(n)}_i - s^{(n)}_i} (\sigma_{s^{(n)}_i} x)
\ d\nu \le - \log c$,
$$
0 \le \sum_{s^{(n)}_i \notin \Lambda_n^N} \int
- \log p_{\mu, S^{(n)}_i - s^{(n)}_i} (\sigma_{s^{(n)}_i} x) \ d\nu
\le
|\Lambda_n \setminus \Lambda_n^N| (-\log c) \leq \epsilon
|\Lambda_n|  (-\log c).
$$

Therefore, by (\ref{decomp}),
\[
- |\Lambda_n| \epsilon/c \leq \int - \log \mu(x(\Lambda_n)) \
d \nu - |\Lambda_n^N| \int I_{\mu}(x) \ d\nu \leq |\Lambda_n|
(\epsilon/c - \epsilon \log c).
\]

By dividing by $|\Lambda_n|$ and letting $n \rightarrow \infty$, we see that

\begin{multline*}
-\epsilon/c + \int I_{\mu}(x) \ d\nu \leq \liminf_{n \rightarrow
\infty}
\frac{\int - \log \mu(x(\Lambda_n))}{|\Lambda_n|} \ d\nu \textrm{ and}\\
\limsup_{n \rightarrow \infty} \frac{\int - \log
\mu(x(\Lambda_n))}{|\Lambda_n|} \ d\nu \leq
(\epsilon/c - \epsilon \log c) +
\int I_{\mu}(x) \ d\nu.
\end{multline*}

Since $\epsilon > 0$ was arbitrary,
\[
\lim_{n \rightarrow \infty} \frac{\int - \log \mu(x(\Lambda_n))}{|\Lambda_n|}
\ d\nu = \int I_{\mu}(x) \ d\nu,
\]
completing the proof.
\end{proof}

\begin{example}
We claim that the pressure representation given in the conclusion of Theorem~\ref{easythm}
fails for the two-dimensional three checkerboard
system $X = C_3^{(2)}$.

Recall that there exists a fully supported uniform Gibbs measure
$\mu$ on $X$; in~\cite{tom-nishant}, it is also shown that $h(\mu) = P(0)$.  This number is positive and in fact has been computed exactly by Lieb~\cite{Lieb} (Lieb computed $P(0)$ for the ``square ice'' model, which is well known to have the same $P(0)$ as $ C_3^{(2)}$~\cite[Example 4.5]{Schmidt1}).

On the other hand, let $x \in \{0,1,2\}^{\zz^2}$ be defined by
$$
x_{(i,j)} = i - j \pmod 3.
$$
Then $x \in X$ and
if $y \in X$ such that $y_{(0,-1)} = x_{(0,-1)}=1$ and $y_{(-1,0)} = x_{(-1,0)} = 2$, then $y_{(0,0)}$ is forced to equal $x_{(0,0)} = 0$.
It follows that $p_\mu(x) = 1$.
In fact, $p_\mu$ takes the value 1 on each point of the (periodic) orbit of $x$.  Thus,
for the periodic point measure $\nu$ on this orbit, $\int I_\mu ~d\nu = 0$.  Since
$A_\Phi \equiv 0$, we have $ \int I_\mu + A_\Phi ~d\nu = 0$, in contrast to $P(\Phi) = P(0) > 0$.

While $X = C_3^{(2)}$ does not satisfy the D-condition, it is topologically mixing~\cite[Proposition 7.2]{Schmidt2} (i.e., there exists $\ell > 0$ such that whenever $S,T$ are
finite subsets of $\mZ^2$ with $d(S,T) \ge \ell$ and $u \in \L_S(X),
v \in \L_T(X)$, then there exists $x \in X$ such that $x(S) = u$
and $x(T) = v$).
\end{example}

\begin{remark}
The first part of the proof of
Theorem~\ref{easythm} implies the result, mentioned at the end of
Section~\ref{pressure}, that if $X_\Phi$  satisfies the D-condition,
then any Gibbs measure for $\Phi$  must be an equilibrium state for
$A_\Phi$: simply integrate equation (\ref{subexperror}) with respect
to $\mu$.

Also, in the case $\Phi \equiv 0$, (\ref{subexperror}) can be viewed as a
uniform version
of the Shannon-MacMillan-Breiman Theorem.
%
\end{remark}

We now apply Theorem~\ref{easythm} to obtain a result which gives
a integral representation of $P(\Phi)$ for {\em every} invariant measure
$\nu$.

\begin{corollary}\label{easycor}
If $\Phi$ is a nearest-neighbor interaction with underlying SFT $X$,
$\mu$ is a Gibbs measure for $\Phi$,\\

\noindent {\rm (B1)} $X$ satisfies the classical D-condition,

\noindent {\rm (B2)}
$\lim_{S \rightarrow \PP} ~p_{\mu, S}(x)= p_\mu(x)$, uniformly over
$x \in \supp(\mu)$, and

\noindent {\rm (B3)} $c_{\mu} > 0$,\\

\noindent then
$$
P(\Phi) = \int I_{\mu}(x) + A_{\Phi}(x) \ d\nu = \int I_{\mu}(x) -
\sum_{i=1}^d \Phi(x(\{0,e_i\})) \ d\nu
$$
for every shift-invariant measure $\nu$ with $\supp(\nu) \subseteq X$.

\end{corollary}

\begin{proof}
This follows immediately from Theorem~\ref{easythm}.
\end{proof}

\begin{corollary}\label{easycor3}
If $\Phi$ is a nearest-neighbor interaction with underlying SFT $X =
X_\Phi$, $\mu$ is a Gibbs measure for $\Phi$,\\

\noindent {\rm (1)} $X$ satisfies SSF, and

\noindent {\rm (2)} $\mu$ satisfies SSM,\\

\noindent then
$$
P(\Phi) = \int I_{\mu}(x) + A_{\Phi}(x) \ d\nu = \int I_{\mu}(x) -
\sum_{i=1}^d \Phi(x(\{0,e_i\})) \ d\nu
$$
for every shift-invariant measure $\nu$ with $\supp(\nu) \subseteq X$.

\end{corollary}

\begin{proof}
This follows from Corollary~\ref{easycor}, Proposition~\ref{SSF_D},
Proposition~\ref{SSM_continuity} and Proposition~\ref{SSF_positive}.
\end{proof}

Corollary~\ref{easycor3} applies to the unique uniform Gibbs measures for the 
hard square shift and the
$k$-checkerboard shift for $k \ge 12$ in dimension $d=2$. 

Next, we state and prove a more difficult version of
Theorem~\ref{easythm}.

\begin{theorem}\label{hardthm}
If $\Phi$ is a nearest-neighbor interaction with underlying $\zz^2$-SFT $X$,
$\mu$ is a Gibbs measure for $\Phi$, $\nu$ is a shift-invariant
measure with $\supp(\nu) \subseteq X$,\\

\noindent {\rm (C1)} $X$ satisfies the block D-condition,

\noindent {\rm (C2)}  $\lim_{S \rightarrow \PP, U \rightarrow
+\infty} p_{\mu, S \cup U}(x) = p_\mu(x)$ uniformly on $\supp(\nu)$, and

\noindent
{\rm (C3)} $p_{\mu}$ is positive over ${\rm \supp}(\nu)$, \\

\noindent
then $P(\Phi) = \int I_{\mu}(x) + A_\Phi(x) \ d\nu = \int I_{\mu}(x) - \sum_{i=1}^d \Phi(x(\{0,e_i\})) \ d\nu$.

\end{theorem}

In assumption (C3),  $p_{\mu}$ is taken, by our convention, to mean the limit of
$p_{\mu,n}$, which exists on $\supp(\nu)$ by assumption (C2).

In comparing Theorems~\ref{easythm} and~\ref{hardthm}, note the
tradeoffs in the assumptions.
\begin{enumerate}
\item
(C1) of ~\ref{hardthm} clearly implies (A1) of~\ref{easythm}.
\item
(C2) of~\ref{hardthm} clearly implies (A2) of~\ref{easythm}.
\item
(A2), (A3) of~\ref{easythm} together imply (C3) of~\ref{hardthm}: by
(A3), $c_{\mu,\nu} > 0$ is a lower bound for
$\{p_{\mu,n}(x): ~ x \in \supp(\nu)\}$, and by (A2), for  $x \in
\supp(\nu)$, $p_{\mu,n}(x)$ approaches $p_{\mu}(x)$, which must also
be bounded from below by $c_{\mu,\nu}$.
\end{enumerate}

As mentioned in the introduction, we believe that Theorem~\ref{hardthm} holds for arbitrary
dimension, but we have stated and proved it only for dimension $d=2$ due to the technicality of a rather
involved geometric decomposition.

We do not have examples where Theorem~\ref{hardthm} applies but Theorem~\ref{easythm} does not.
While the proof of Theorem~\ref{hardthm} is considerably more difficult, we have included it
because it uses an interesting reordering of vertices in the decomposition of $\mu(x(B_n))$ that is used to
compensate for what appears to be a weaker assumption.
\medskip

{\em Proof of Theorem~\ref{hardthm}}.

As in the proof of Theorem~\ref{easythm}, since the D-condition
holds on $X$, it suffices to show that $\lim_{n \rightarrow \infty}
\int \frac{- \log \mu(x(B_n))}{(2n+1)^{2}} \ d\nu = \int I_{\mu} \
d\nu$. For our purposes, it will be more convenient to show
that $\lim_{n \rightarrow \infty} \int \frac{- \log
\mu(x(K_n))}{n^{2}} \ d\nu = \int I_{\mu} \ d\nu$, where
$K_n = [0,n+1]^{2}$. Clearly,
since $\frac{|K_n|}{n^{2}} \rightarrow 1$, this
still suffices.

Analogous to the proof of Theorem~\ref{easythm},  we will decompose
$\mu(x(K_n))$ as a product of conditional probabilities. However, we
no longer have a positive lower bound on $p_{\mu, S}(x)$ over $x \in
\supp(\nu)$ and finite subsets $S$ of $\PP$. By using a more
complicated decomposition of $K_n$ and the block D-condition, we
obtain a much weaker (but still strong enough for our purposes)
lower bound on the conditional probabilities involving sites near
the boundary of $K_n$. The following lemma gives this bound (which
we state for arbitrary dimension $d$ because it seems to be of independent interest.)

\begin{lemma}\label{weakbd}
Let $\Phi$ be a nearest-neighbor interaction with underlying SFT $X
=X_\Phi$.  Suppose that $S,T,U$ are finite subsets of $\mathbb{Z}^d$
such that $S \subseteq T \subseteq U$, $U$ is connected, $S \cap \underline{\partial} U = \varnothing$, and $[v]
\cap [w] \neq \varnothing$ for all $v \in \L_S(X)$ and $w \in \L_{U
\setminus T}(X)$. Then, for any $x \in \L_{U}(X)$,
\[
\mu(x(U \setminus T) \ | \ x(\underline{\partial} U)) \geq \gamma^{-(|U| - |S|)},
\]
where $\gamma = |\A| e^{2dL - \ell}$ 
for $\ell$ and $L$ the
minimum and maximum, respectively, of finite values of $\Phi$.
\end{lemma}

\begin{proof}

Consider any such $\Phi$, $S$, $T$, and $U$. Then for every $v \in \L_S(X)$ and $w \in \L_{U \setminus T}(X)$,
there exists $u_{v,w} \in \L_{T \setminus S}(X)$ so that $vu_{v,w}w \in \L(X)$. Fix $x \in \L_{U}(X)$ and
denote $y = x(U \setminus (T \cup \underline{\partial} U))$, $\delta = x(\underline{\partial} U)$ and $\bar{u} = u_{v,y\delta}$. Since $\mu$ is a Gibbs measure,
\begin{multline*}
\mu(y \ | \ \delta) = \frac{\sum_{z \in \L_{U \setminus \underline{\partial} U}(X)
\textrm{ s.t. } z\delta \in \L(X) \textrm{ and } z(U \setminus
(T \cup \underline{\partial} U)) = y}
~~e^{-U^{\Phi}(z\delta)}}{\sum_{z \in \L_{U \setminus \underline{\partial} U}(X)
\textrm{ s.t. }
z\delta \in \L(X)} e^{-U^{\Phi}(z\delta)}}\\
\geq \frac{\sum_{v \in \L_S(X)} e^{-U^{\Phi}(v \bar{u}
y\delta)}}{\sum_{v \in \L_S(X)} |\A|^{|U| - |S|} \max_{u \in \L_{U
\setminus (S \cup \underline{\partial} U)}(X)
\textrm{ s.t. } uv\delta \in \L(X)}~~ e^{-U^{\Phi}(uv\delta)}}\\
\geq \frac{\sum_{v \in \L_S(X)} e^{-U^{\Phi}(v) - 2d L (|U| - |S|)}}
{\sum_{v \in \L_S(X)} |\A|^{|U| - |S|} e^{-U^{\Phi}(v) - \ell (|U| -
|S|)}} = (|\A| e^{2dL - \ell})^{-(|U| - |S|)} = \gamma^{-(|U| -
|S|)}.
\end{multline*}

\end{proof}

We decompose $K_n$ into geometric shapes, and for those near the
boundary will use Lemma~\ref{weakbd} to show that the conditional
probability of filling the shape in a certain way cannot be too
small.

By (C2), for any $\epsilon > 0$, there exists $k := k_{\epsilon}$ so
that for any finite sets $S,U$ with $\PP_{k}
\subseteq S \subset \PP$ and $U \subseteq (B_k \cup \PP)^c$,
$|p_{\mu, S \cup U}(x) - p_{\mu}(x)| < \epsilon$ for all $x \in
\supp(\nu)$. (Equivalently, we could just require that $S \cup U$
contain $\PP_k$ and be contained in $B_k^c \cup \PP$.)

Since $X$ satisfies the block D-condition, there exists a sequence
$R_n$ of integers s.t. $\frac{R_n}{n} \rightarrow 0$ and, for any
rectangular prism $P = \prod_i [1,n_i]$, any $r \ge  R_{\max n_i}$,
and any configurations $v \in \L_P(X)$ and $w \in \L_S(X)$ for any
finite $S \subset {(\prod_i [1 - r, n_i + r])^c}$, we have $[v] \cap
[w] \neq \varnothing$ (see Lemma~\ref{blockD}). We may assume
without loss of generality that $R_n$ is non-decreasing by
redefining each $R_n$ as $\max_{i \leq n} R_i$.

Our construction requires a parameter $m:= m_n$, which we choose to
be equal to $\lfloor \sqrt{nR_n} \rfloor$, and so
$\frac{R_n}{n} \rightarrow 0$, $\frac{m_n}{n} \rightarrow 0$,
and $\frac{R_n}{m_n} \rightarrow 0$..
We only consider $n$
large enough so that $m > 2k$ and $\frac{R_n}{n} < 1$.
 Define sets\\

\noindent
$C_0 = (\underline{\partial} K_n)$,\\ 

\noindent $C_1 = \{(i,j) \ : \ 1 \leq j \le m, 1 \leq i \leq k(m+1-j)\}$,

\noindent $C_2 = \{(i,j) \ : \ 1 \leq j \le m, k(m+1-j) < i
< n - kj\}$,

\noindent
$C_3 = \{(i,j) \ : \ 1 \leq j \le m,
 n - kj \leq i \leq n\}$,\\

\noindent
$C_{3t + 1} = C_1 + (0,tm)$ for all $ 1 \le t \le \lfloor \frac{n}{m} \rfloor - 2$,

\noindent
$C_{3t + 2} = C_2 + (0,tm)$ for all $1 \le t \le  \lfloor \frac{n}{m} \rfloor - 2$,

\noindent
$C_{3t + 3} = C_3 + (0,tm)$ for all $1 \le t \le  \lfloor \frac{n}{m} \rfloor - 2$, and\\

\noindent
$C_{3 \lfloor \frac{n}{m} \rfloor - 2} = [1,n] \times [m \lfloor \frac{n}{m} \rfloor - m + 1, n]$.\\

%
%
%

\begin{figure}[h]
\centering
\includegraphics[scale=0.4]{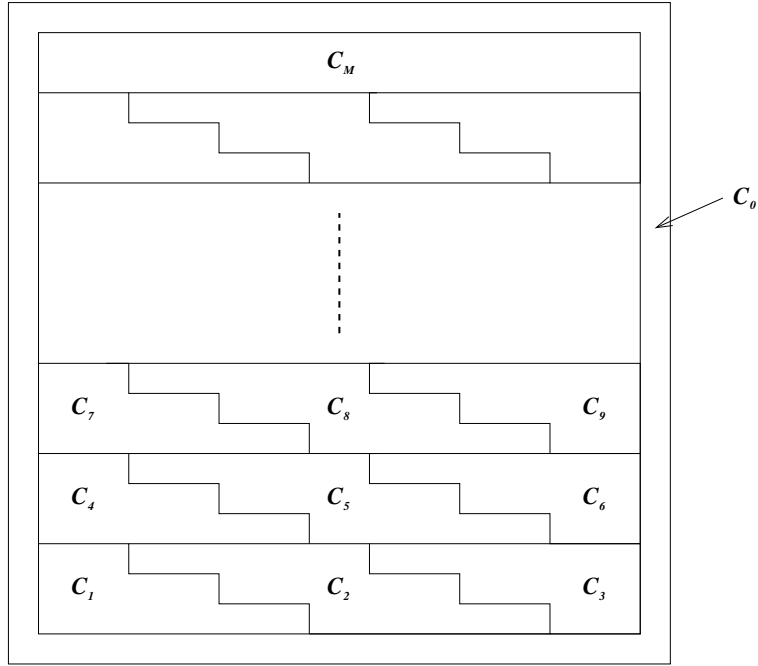}
\caption{Decomposing $K_n$ (here, $M = 3 \lfloor \frac{n}{m} \rfloor
-2$)} \label{pressurepic1}
\end{figure}

As illustrated in Figure~\ref{pressurepic1}, $C_0$ is the inner
boundary of $K_n$. For each $t$, $C_{3t + 1}$, $C_{3t + 2}$, and
$C_{3t + 3}$ form a partition of the strip $[1,n] \times [tm + 1, (t+1)m]$ of height $m$,
made of two (discrete) trapezoids and a (discrete) parallelogram.
And $C_{3 \lceil \frac{n}{m} \rceil - 2}$ is simply a single ``leftover'' strip at the top of
$[1,n]^2 = K_n \setminus \underline{\partial}K_n$. For every $t \in [1,3 \lceil
\frac{n}{m} \rceil - 2]$, we define $D_t = \bigcup_{s =
0}^{t-1} C_s$, and define $D_0 = \varnothing$.
Note that by the choice of $m$, for large $n$, the bulk of $K_n$ is comprised of
$\{C_i; ~ i = 2 \pmod 3 \} $.

We decompose $\mu(x(K_n))$ as
\begin{equation}\label{bigdecomp}
\prod_{i=0}^{3 \lceil \frac{n}{m} \rceil -2} \mu(x(C_i) \ | \ x(D_i)).
\end{equation}

We begin by giving a lower bound for $\mu(x(C_0) \ | \ x(D_0)) =
\mu(x(C_0))$. For any $n$, letting $R := R_n$, then by definition $S
= [R+1,n-R]^2$, $T = K_n \setminus C_0$ and $U =  [-R,n+R]^2$
satisfy the hypotheses of Lemma~\ref{weakbd}. There clearly exists
$\delta_n \in \L_{\underline{\partial} U}(X)$ s.t. $\mu(\delta_n)
\ge |\A|^{-|\underline{\partial} U|}$. For any $x \in X$, by
definition of $R$, there exists $y \in \L_{U \setminus (K_n \cup \underline{\partial} U )}$
s.t. $x(C_0) y\delta_n \in \L(X)$.
So, by Lemma~\ref{weakbd},
\[
\mu\left(x(C_0) y \ | \ \delta_n\right) \geq \gamma^{-(|U| - |S|)}
\geq \gamma^{{-4(2R+1)(n+2R+1)}}.
\]
Therefore,
\begin{multline}\label{boundarybd}
\mu(x(C_0)) \geq \mu(x(C_0) y \delta_n) = \mu(\delta_n)
\mu(x(C_0) y \ | \ \delta_n) \\
\geq 
|\A|^{-4(n+2R+1)} \gamma^{-4(2R+1)(n+2R+1)} \geq \gamma^{-4(2R+2)(n+2R+1)}.
\end{multline}
(here the last inequality follows from $\gamma > |\A|$.)

The final factor of (\ref{bigdecomp}) is easy to bound from below.
Note that the sets $S = T = \varnothing$ and $U = C_{3 \lfloor
\frac{n}{m} \rfloor - 2} \cup \partial C_{3 \lfloor
\frac{n}{m} \rfloor - 2}$ satisfy the hypotheses of
Lemma~\ref{weakbd}. Then,

\begin{multline}\label{lasttermbd}
\mu(x(C_{3 \lfloor \frac{n}{m} \rfloor - 2}) \ | \
x(D_{3 \lfloor \frac{n}{m} \rfloor - 2})) \\
= \mu(x(C_{3 \lfloor \frac{n}{m} \rfloor - 2}) \ | \ x(\partial C_{3
\lfloor \frac{n}{m} \rfloor - 2})) = \mu(x(U \setminus T) \ | \
x(\underline{\partial}U )) \geq \gamma^{-(|U| - |S|)} \geq
\gamma^{-(n+2)(2m+2)}.
\end{multline}
(here the first equality follows from the fact that $\mu$ is an MRF.)

We next deal with the terms in (\ref{bigdecomp}) of the form
$\mu(x(C_{3t + 1}) \ | \ x(D_{3t + 1}))$. We wish to apply
Lemma~\ref{weakbd} for $U = [0,n+1] \times [tm, n+1]$, $T = U
\setminus (\underline{\partial} U \cup C_{3t + 1})$, and
$S = ([km + R + 1, n - R] \times [tm + R + 1, (t + 1)m]) \cup ([R + 1, n - R] \times [(t + 1) m + R + 1, n - R])$.
(See Figure~\ref{pressurepic2}.)

\begin{figure}[h]
\centering
\includegraphics[scale=0.23]{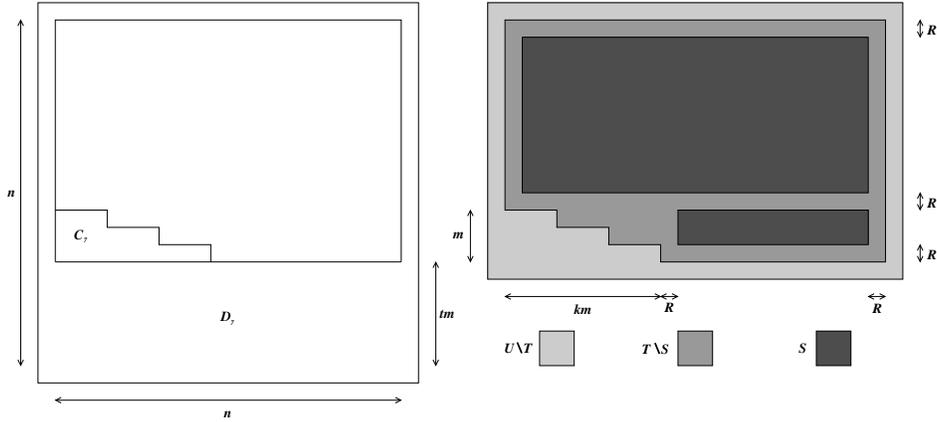}
\caption{$S$, $T$, and $U$ for $\mu(x(C_{3t+1}) \ | \ x(D_{3t+1}))$}
\label{pressurepic2}
\end{figure}

To check that Lemma~\ref{weakbd} can be used, we must show that for
any configurations $v \in \L_S(X)$ and $w \in
\L_{\underline{\partial} U \cup C_{3t + 1}}(X)$, $[v] \cap [w] \neq
\varnothing$. This requires two applications of the block
D-condition. Write $v$ as a concatenation $pq$ for $p \in \L_{[km +
R + 1, n - R] \times [tm + R + 1, (t + 1)m]}(X)$ and $q \in \L_{[R + 1, n -
R] \times [(t + 1)m + R + 1, n - R]}(X)$. By the block D-condition,
there exists a configuration $w'$ which extends $w$ and $p$. A
second application of the block D-condition gives a configuration
$w''$ that extends $q$ and $w'$ and hence $p,q$ and $w$. Thus, $[p]
\cap [q] \cap [w] = [v] \cap [w] \neq \varnothing$.

Since $\mu$ is an MRF, Lemma~\ref{weakbd} implies that
\begin{multline}\label{1mod3}
\mu(x(C_{3t + 1}) \ | \ x(D_{3t+1})) =
\mu(x(C_{3t + 1}) \ | \ x(\underline{\partial} ([0,n+1] \times [tm, n+1]))) \geq \gamma^{-(|U| - |S|}) \\
\geq \gamma^{-(km^2 + 5(R+1)(n+2))}.
\end{multline}

We next deal with the terms in (\ref{bigdecomp}) of the form $\mu(x(C_{3t + 3}) \ | \ x(D_{3t + 3}))$. We wish to apply Lemma~\ref{weakbd} for $U = (([0,n+1] \times
[tm, n+1]) \cup C_{3t + 3}) \cup \partial C_{3t + 3}$, $T = U \setminus (\underline{\partial} U \cup C_{3t + 3})$, and
$S = [R + 1, n - R] \times [(t + 1)m + R + 1, n - R]$. (See Figure~\ref{pressurepic3}.)

\begin{figure}[h]
\centering
\includegraphics[scale=0.23]{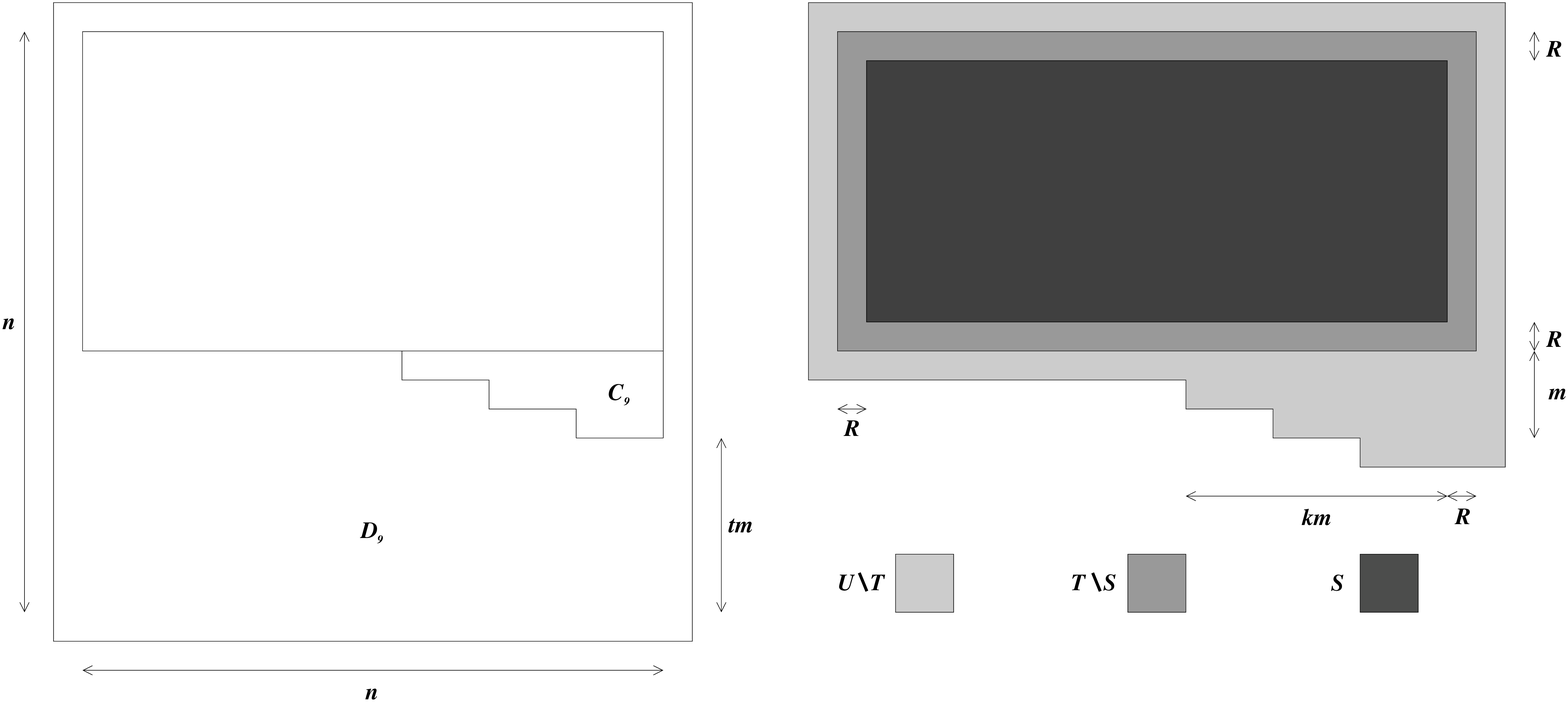}
\caption{$S$, $T$, and $U$ for $\mu(x(C_{3t+3}) \ | \ x(D_{3t+3}))$}
\label{pressurepic3}
\end{figure}

To check that Lemma~\ref{weakbd} can be used, we must show that for
any configurations $v \in \L_S(X)$ and $w \in \L_{\underline{\partial} U \cup
C_{3t + 3}}(X)$, $[v] \cap [w] \neq \varnothing$. This is a
straightforward application of the block D-condition.

Since $\mu$ is an MRF, we can use Lemma~\ref{weakbd} to show that
\begin{multline}\label{3mod3}
\mu(x(C_{3t + 3}) \ | \ x(D_{3t + 3})) = \mu(x(C_{3t + 3}) \ | \ x(\underline{\partial} U)) \geq \gamma^{-(|U| - |S|)} \\
\geq \gamma^{-(km^2 + 4(R+1)(n+2))}.
\end{multline}

It remains to deal with factors of the form $\mu(x(C_{3t + 2}) \ | \
x(D_{3t + 2}))$. For each $t$, denote the sites of $C_{3t+2}$ in
lexicographic order as $s^{(3t+2)}_i$, $1 \leq i \leq |C_{3t + 2}|$,
for each such $i > 1$, define $S^{(3t + 2)}_i = \bigcup_{j = 1}^{i-1} \{s^{(3t+2)}_j\}$,
and define $S^{(3t + 2)}_1 = \varnothing$. We first
decompose $\mu(x(C_{3t + 2} \ | \ D_{3t + 2}))$ as
\begin{equation}\label{littledecomp}
\prod_{i = 1}^{|C_{3t + 2}|} \mu\left(x(s^{(3t+2)}_i) \ | \
x(S^{(3t+2)}_i \cup D_{3t + 2})\right).
\end{equation}


Let $t \ge 1$.  For every $i$,
\begin{multline*}
\mu\left(x(s^{(3t+2)}_i) \ | \ x(S^{(3t+2)}_i \cup D_{3t + 2})\right)
= p_{\mu, (S^{(3t+2)}_i \cup D_{3t + 2}) - s^{(3t+2)}_i} (\sigma_{s^{(3t+2)}_i} x).
\end{multline*}
The discrete parallelogram structure of $C^{(3t+2)}$ guarantees that
each set \newline $S^{(3t+2)}_i \cup D_{3t + 2}  - s^{(3t+2)}_i$
contains $\PP_k$ and is contained in $\PP \cup B_k^c$.
By definition of $k$, we then have
\begin{equation}\label{closesites2}
\left|p_{\mu, (S^{(3t+2)}_i \cup D_{3t + 2}) - s^{(3t+2)}_i} (\sigma_{s^{(3t+2)}_i} x)
- p_{\mu}(\sigma_{s^{(3t+2)}_i} x) \right| < \epsilon.
\end{equation}

We claim that (\ref{closesites2}) also holds for $t=0$ (and every $i$).
In this case, $S^{(2)}_i \cup D_2  - s^{(2)}_i$
need not contain $\PP_k$.  However,
$S^{(2)}_i \cup D_2   \cup [0,n+1]\times [-k,-1] - s^{(2)}_i$  does contain $\PP_k$ (and is contained in $\PP \cup B_k^c$). Moreover, since $\mu$ is an MRF and
$D_2 \supseteq C_0 = \underline{\partial} K_n$, we have
$$
p_{\mu, S^{(2)}_i \cup D_2 - s^{(2)}_i}(\sigma_{s^{(2)}_i} x) = p_{\mu, S^{(2)}_i \cup D_2  \cup [0,n+1]\times [-k,-1] - s^{(2)}_i}(\sigma_{s^{(2)}_i } x)
$$
which is within $\epsilon$ of
$p_{\mu}(\sigma_{s^{(2)}_i} x)$, proving the claim.

It follows from (C2) that
$p_{\mu}$ is the uniform limit of continuous functions on
$\supp(\nu)$; since, by (C3), it is also positive on ${\rm
supp}(\nu)$, it has a lower bound $c
> 0$ there. We can therefore integrate with respect to $\nu$ to see
that
\begin{equation}\label{closesites}
\left| \int - \log p_{\mu, (S^{(3t+2)}_i \cup D_{3t + 2} ) - s^{(3t+2)}_i} (\sigma_{s^{(3t+2)}_i} x) \ d\nu - \int I_{\mu}(x) \ d\nu \right| < \epsilon c^{-1}.
\end{equation}

We now combine (\ref{bigdecomp}), (\ref{boundarybd}),
(\ref{lasttermbd}), (\ref{1mod3}), (\ref{3mod3}),
(\ref{littledecomp}), and (\ref{closesites}) to see that
\begin{multline*}
\left|\int - \log \mu(x(K_n)) \ d\nu - \int I_{\mu}(x) \
d\nu \left(\sum_{t=0}^{\lfloor \frac{n}{m} \rfloor - 2} |C_{3t + 2}|\right)\right| \\
\leq n^2 \epsilon c^{-1} + \frac{n}{m} (9(R+1)(n+2) + 2km^2) \log \gamma
+ (4(2R+2) + 2m + 2)(n+2R+1) \log \gamma \\
\leq n^2 \epsilon c^{-1} + n(n+2R+1) \log \gamma
\left(\frac{9(R+1)}{m} + \frac{2km}{n+2} + \frac{4(2R+2) + 2m +
2}{n} \right).
\end{multline*}

Note that $\sum_{t} |C_{3t + 2}|
 \geq n^2 - 2nm - 2\frac{n}{m} m(km + 1) = n^2 - 2n((k+1)m+1)$. Therefore,
\begin{multline*}
\left| \int \frac{- \log \mu(x(K_n))}{n^2} \ d\nu - \int I_{\mu}(x)
\ d\nu \right| \leq \frac{2(k+1)(m+1)}{n} \left| \int I_{\mu}(x) \
d\nu \right| + \epsilon c^{-1} \\ + \log \gamma \left(1 +
\frac{2R}{n}  +\frac{1}{n}  \right) \left(\frac{9(R+1)}{m} +
\frac{2km}{n+2} + \frac{4(2R+2) + 2m + 2}{n} \right).
\end{multline*}

Recalling that $m = \lfloor \sqrt{nR} \rfloor$, $\frac{R}{n} \rightarrow 0$,
and $k$ is a constant, the right-hand side of this inequality approaches
$\epsilon c^{-1}$  as $n \rightarrow \infty$.  Therefore,
\begin{multline*}
-\epsilon c^{-1} + \int I_{\mu}(x) \ d\nu \leq \liminf_{n \rightarrow \infty}
\frac{\int - \log \mu(x(K_n))}{n^2} \ d\nu \textrm{ and}\\
\limsup_{n \rightarrow \infty} \frac{\int - \log \mu(x(K_n))}{n^2}
\ d\nu \leq \epsilon c^{-1} + \int I_{\mu}(x) \ d\nu.
\end{multline*}

By letting $\epsilon \rightarrow 0$, we see that

\[
\lim_{n \rightarrow \infty} \frac{\int - \log \mu(x(K_n))}{n^2}\ d\nu = \int I_{\mu}(x) \ d\nu,
\]
completing the proof. \hspace{ 3.45 in} $\square$

\section{Pressure Approximation schemes}
\label{approx_alg}

In this section, we derive, as a consequence of our pressure representation results, an
algorithm for
approximating the pressure of a shift-invariant nearest-neighbor
Gibbs interaction $\Phi$.
For this, we assume that the exact values of $\Phi$ are known;
it would be impossible to hope for a bound on computation time for $P(\Phi)$
if $\Phi$ were uncomputable or very hard to compute.
The main idea of this section comes from Gamarnik and Katz~\cite[Corollary 1]{GK}.

We would like to apply our main results to shift-invariant measures
$\nu$ which are as easy as possible to integrate against, for
instance the atomic measure supported on a periodic orbit. In
general, a $\mathbb{Z}^d$ SFT $X$ need not have any periodic points
(\cite{berger}). However, any nearest neighbor SFT $X$ that
satisfies SSF must have a periodic point: choose any $b \in \A$ and let $a
\in \A$ such that $b^{N_0}a^{\{0\}}$ is locally admissible in $X$,
then the point $z$ defined by $z_v = a$ if $\sum_i~ v_i$ is
even and $z_v = b$ if $\sum_i~ v_i$ is odd is periodic and in $X$.

\begin{proposition}
\label{approx_scheme}
 Let $\Phi$ be a nearest neighbor $\zz^d$
interaction and $X = X_\Phi$.  Assume that
\begin{enumerate}
\item[(i)] $X$ satisfies SSF,
\item[(ii)] $\Phi$ satisfies SSM at exponential rate.
\end{enumerate}
Then there is an algorithm to compute $P(\Phi)$ to within $\epsilon$
in time  $e^{O((\log \frac{1}{\epsilon})^{d-1})}$.
\end{proposition}

Note that in the case $d=2$ this gives a polynomial-time
approximation scheme.

\begin{proof}

Let $\mu$ be the unique Gibbs state corresponding to $\Phi$.

Let $z$ be a periodic point, which exists by SSF, and $\nu$ the
shift-invariant atomic measure supported on the orbit of $z$.
The assumptions of Theorem~\ref{easythm} are satisfied:
assumptions (A1) and (A3) follow from SSF by
Propositions~\ref{SSF_D} and~\ref{SSF_positive}, and assumption (A2)
follows from SSM and Proposition~\ref{SSM_continuity}.

We conclude from  Theorem~\ref{easythm} that
$$
P(\Phi) = \int (I_\mu + A_\Phi) d\nu = (1/|D|)(\sum_{v \in D} -\log
p_\mu(\sigma^{ v}(z)) - \sum_{v\in D, ~v' \sim v} \Phi(z(\{v,v'\})))
$$
(here, $D \subset \mZ^2$ is a fundamental domain for $z$).

Since we assume that the exact values of the interaction $\Phi$ are
known, it suffices to compute the desired approximations to
$p_\mu(x)$ for  all $x = \sigma^v(z), ~v \in D$.  We may assume
$v=0$ (the proof is the same for all $v$).

Recall the notation from the proof of
Proposition~\ref{SSM_continuity}:
$$
S_n = B_n \setminus \PP_n, ~\partial S_n = U_n \cup C_n
$$
where
$$
U_n = (\partial S_n) \cap \PP, 
C_n = \partial S_n \setminus U_n.
$$
Note that no site in $C_n$ neighbors one in $U_n$. So, for any locally admissible
configurations on $C_n$ and $U_n$, their concatentation is locally admissible,
and therefore globally admissible by SSF. Then, by Propositions~\ref{full} and
~\ref{SSF_D}, any such concatenation has positive $\mu$-measure as well.
Therefore, we may use the fact that $\mu$ is an MRF to represent $p_{\mu}$ as a weighted average:
$$
p_\mu(z) = \sum_{\mbox{locally admissible }\delta \in \A^{C_n}}
~\mu(z(0) \ | \ z(U_n) \delta)\mu(\delta).
$$

Let $\delta^{z(0),n}$ achieve $\max \mu(z(0) \ | \ z(U_n) \delta)$
and $\delta_{z(0), n}$ achieve $\min \mu(z(0) \ | \ z(U_n) \delta)$
over all locally admissible $\delta \in \A^{C_n}$. Clearly,
$$
\mu(z(0) \ | \ z(U_n) \delta_{z(0), n}) \le p_\mu(z) \le \mu(z(0) \
| \ z(U_n) \delta^{z(0), n}).
$$
By SSM at exponential rate, there are constants $C, \alpha > 0$ such
that these upper and lower bounds on $p_\mu(z)$ differ by at most
$Ce^{-\alpha n}$.

This gives sequences of upper and lower bounds on  $p_\mu(z)$ with
accuracy $e^{-\Omega(n)}$.  For $\delta \in \A^{C_n}$, the time to
compute $\mu(z(0) \ | \ z(U_n) \delta)$ is $e^{O(n^{d-1})}$ since
this is the ratio of two probabilities of configurations of size
$O(n^{d-1})$, each of which can be computed using
the transfer matrix method from~\cite[Lemma 4.8]{MP2} in time $e^{O(n^{d-1})}$. Since $|\A^{C_n}| =
e^{O(n^{d-1})}$, the total time to compute the upper and lower
bounds is $e^{O(n^{d-1})}e^{O(n^{d-1})} = e^{O(n^{d-1})}$.
\end{proof}

There are many sufficient conditions for SSM at exponential rate
(for instance, see the discussion in~\cite{MP2}).

We remark that Corollary 4.13 of~\cite{MP2} gives an algorithm to
compute $P(\Phi)$ that is less efficient, in that the approximation
to within $\epsilon$ requires time $e^{O((\log
\frac{1}{\epsilon})^{(d-1)^2})}$. However, it applies more generally
than Proposition~\ref{approx_scheme} in that it requires only
assumption (ii) above and not assumption (i).

Finally, we note that, for a nearest-neighbor interaction $\Phi$,
the algorithm here to compute $P(\Phi)$ also computes $h(\mu)$ for
any Gibbs measure corresponding to $\Phi$.

\section{Connections with Thermodynamic Formalism} \label{connections}

In this section, we connect Corollary~\ref{easycor} with results
from Ruelle's thermodynamic formalism.  We begin by taking a new
look at Corollary~\ref{easycor}.

Consider the following probability-vector-valued functions:
\[
\hat{p}_{\mu,n}(x) := \mu(y(0)  = \cdot \ | \ y(\PP_n) = x(\PP_n))
\]
and
\[
\hat{p}_{\mu}(x) := \mu(y(0)  = \cdot \ | \ y(\PP) = x(\PP)).
\]
Note that these functions do not depend on $x_0$ ($\hat{p}_{\mu,n}$
is similar in spirit to the function $\hat{p}^n_{\mu}$, which was
introduced and used in the proof of
Proposition~\ref{SSM_continuity}).

Note that $\hat{p}_{\mu}$ is defined only $\mu$-a.e. By Martingale
convergence, $\hat{p}_{\mu,n}$ converges to $\hat{p}_\mu$ for
$\mu$-a.e. $x \in \supp(\mu)$.

The following result relates uniform convergence of
$\hat{p}_{\mu,n}$ with a continuity property of $\hat{p}_\mu$.

\begin{definition}
A function $g$ is {\bf past-continuous} on a shift space $X$ if it
is continuous on $X$ and, for all $x \in X$, $g(x)$ depends only on $x(\PP)$:
if $x,y \in X$ and $x(\PP) = y(\PP)$, then $g(x) = g(y)$.
\end{definition}

\begin{proposition}
\label{unif_conv_cts} Let $\mu$ be a stationary measure.
\begin{enumerate}
\item[(i)]
$\hat{p}_{\mu,n}$ converges uniformly on $\supp(\mu)$ iff
$\hat{p}_\mu$ is past-continuous on $\supp(\mu)$ (i.e.,
$\hat{p}_\mu$ agrees with a past-continuous function ($\mu$-a.e.)).
\item[(ii)]
If $\hat{p}_{\mu,n}$ converges uniformly on $\supp(\mu)$, then
$\lim_{S \rightarrow \PP} ~p_{\mu,S} = p_\mu(x)$, uniformly on
$\supp(\mu)$.
\end{enumerate}
\end{proposition}

\begin{proof}

For finite $S \subset \mathcal{P}$, $a \in \A$ and $x \in
\supp(\mu)$, we can write
\begin{equation}
\label{eqn_wtd} \mu(y(0) = a \ |  \ y(S) = x(S)) = \frac{1}{\mu([
x(S)])} \int_{[x(S)]} ~ p_\mu(y(0) = a \ | \ y(\PP) = x(\PP))
d\mu(y).
\end{equation}

\noindent
(i), $\Leftarrow$: If $\hat{p}_\mu$ is past-continuous on $\supp(\mu)$, then we can
take the integrand in (\ref{eqn_wtd}) to be a continuous, and
therefore uniformly continuous, function of $y(\PP), ~y \in [x(S)]$.
Thus, taking $S = \PP_n$ we get uniform convergence of
$\hat{p}_{\mu,n}$.\\

\noindent
(i), $\Rightarrow$: If $\hat{p}_{\mu,n}$ converges uniformly on $\supp(\mu)$, then its
limit is a uniform limit of past-continuous functions on $\supp(\mu)$ and thus is
past-continuous on $\supp(\mu)$. By Martingale convergence, this limit agrees with
$\hat{p}_\mu$ ($\mu$-a.e.).\\

\noindent
(ii): By (i), we may assume that $\hat{p}_\mu$ is past-continuous. Take
$a = x(0)$ in (\ref{eqn_wtd}). Given $\epsilon
> 0$, for sufficiently large $n$, if $S$ is a finite set satisfying
$\PP_n \subset S \subset \mathcal{P}$, then for all $x \in
\supp(\mu)$, the integrand in (\ref{eqn_wtd}) is within $\epsilon$
of $p_\mu(x)$, and so $p_{\mu,S}(x)$ is within $\epsilon$ of
$p_\mu(x)$.

\end{proof}

\begin{corollary}\label{easycor4}
If $\Phi$ is a nearest-neighbor interaction with underlying SFT $X$,
$\mu$ is a Gibbs measure for $\Phi$,\\

\noindent {\rm (D1)} $X$ satisfies the classical D-condition,

\noindent {\rm (D2)} $\hat{p}_\mu$ is past-continuous on $X$, and

\noindent {\rm (D3)} $c_{\mu} > 0$,\\

\noindent then
$$
P(\Phi) = \int I_{\mu}(x) + A_{\Phi}(x) \ d\nu = \int I_{\mu}(x) -
\sum_{i=1}^d \Phi(x(\{0,e_i\})) \ d\nu
$$
for every shift-invariant measure $\nu$ with $\supp(\nu) \subseteq
X$.

\end{corollary}

\begin{proof}
This follows immediately from Corollary~\ref{easycor} and
Proposition~\ref{unif_conv_cts}.
\end{proof}

Next, we will show how Ruelle's thermodynamic formalism~\cite{Ruelle}
can be applied to obtain a result similar to
Corollary~\ref{easycor4}. For this, we need to give a definition
of interactions more general than nearest-neighbor.

\begin{definition}
An \textbf{interaction} is a shift-invariant function $\Phi$ from $\A^*$
to $\mR \cup \{\infty\}$ which takes on the value
$\infty$ for only finitely many configurations on shapes containing
$0$. An interaction is \textbf{finite-range} if there exists $N$ so
that if $\Phi(w) \neq 0$, then $w$ has shape with diameter at most
$N$.
\end{definition}

We remark, without proof, that the results of this paper extend to
finite-range interactions, in particular to interactions which are
non-zero only on vertices and edges (such interactions are often
also called nearest-neighbor interactions, but form a slightly more
general class than the interactions that we have called
nearest-neighbor in this paper).

\begin{definition}
An interaction is called  \textbf{summable} if
\[
\sum_{\Lambda \subset \mathbb{Z}^d, \Lambda \ni 0, |\Lambda| <
\infty} \frac{1}{|\Lambda|} \max_{x \in \A^\Lambda: ~\Phi(x) \ne
\infty} |\Phi(x)| < \infty.
\]
An interaction is called   \textbf{absolutely summable} if
\[
\sum_{\Lambda \subset \mathbb{Z}^d, \Lambda \ni 0, |\Lambda| <
\infty} \max_{x \in \A^\Lambda: ~\Phi(x) \ne \infty} |\Phi(x)| <
\infty.
\]
\end{definition}

The underlying SFT corresponding to an interaction $\Phi$ is defined
as follows.
$$
X_\Phi = \{x \in \A^{\mathbb{Z}^d} \ : \ \Phi(x(S)) \neq \infty
\textrm{ for all finite } S \subseteq \mathbb{Z}^d\}.
$$
The role of configurations on finite subsets with $\Phi(x) = \infty$
corresponds to forbidden configurations and defines the SFT
$X_\Phi$.  This is equivalent to Ruelle's setting where one starts
with a background SFT $X$ and only considers finite-valued
interactions defined on configurations that are not forbidden.

A continuous function $A_\Phi$ and underlying SFT $X_\Phi$ can be
associated to any summable interaction $\Phi$, in a similar
fashion as what was done for nearest-neighbor interactions in Section~\ref{MRF}:
$$
A_{\Phi}(x) := -\sum_{\Lambda \subseteq \mathcal{P} \cup \{0\},
\Lambda \ni 0, |\Lambda| < \infty, ~ \Phi(x(\Lambda)) \ne \infty}
\Phi(x(\Lambda)).
$$

Let $X$ be a nonempty SFT and let $I(X)$ denote the set of summable
interactions $\Phi$  with $X_\Phi = X$. Ruelle shows
in~\cite[Section 3,2]{Ruelle} that for any continuous function $f$
on $X$, there exists a summable interaction $\Phi$ such that $X_\Phi
= X$ and $f = A_{\Phi}$.  That is, the linear mapping $\Phi \mapsto A_\Phi$
from $I(X)$ to $C(X)$ is surjective.  However, the restriction of
this mapping to the set of absolutely summable interactions is not
surjective.

For an absolutely summable interaction, the concepts of energy
function, partition function and Gibbs measure are defined
analogously to those in Section~\ref{MRF}. For details,
see~\cite{Ruelle}. For us, a Gibbs measure is shift-invariant by
definition, as in the nearest-neighbor case.

Gibbs measures and equilibrium states are intimately connected by
the following two standard theorems (these were mentioned at
the end of Section~\ref{pressure} for the special case of
nearest-neighbor interactions).  Proofs of both can be found in
~\cite{Ruelle}.

\begin{theorem}{\rm (\cite{LR})}\label{Lanford-Ruelle}
If $\Phi$ is an absolutely summable interaction, then any
equilibrium state for $A_{\Phi}$ on $X_\Phi$ is a Gibbs measure for
$\Phi$.
\end{theorem}

\begin{theorem}{\rm (\cite{dob})}\label{Dobrushin}
If $\Phi$ is an absolutely summable interaction whose underlying SFT
$X_\Phi$ satisfies the D-condition, then any Gibbs measure for
$\Phi$ is an equilibrium state for $A_{\Phi}$ on $X_\Phi$.
\end{theorem}

\begin{definition}{\rm (\cite{Ruelle})}
Absolutely summable interactions $\Phi$ and $\Phi'$ with the same underlying SFT $X$ are
called \textbf{physically equivalent} if $A_{\Phi}$ and $A_{\Phi'}$
have a common equilibrium state.
\end{definition}

The following result is Proposition 4.7(b) from \cite{Ruelle}.

\begin{theorem}\label{ruellethm}
If $X$ is an SFT that satisfies the D-condition, and  if $\Phi,
\Phi'$ are physically equivalent absolutely summable interactions
with underlying SFT $X$, then $\int (A_{\Phi} - A_{\Phi'}) \ d\nu$
is constant for all shift-invariant measures $\nu$ on $X$.
\end{theorem}

In order to make a connection with Corollary~\ref{easycor4}, we need
the following.

\begin{lemma}\label{ruellelem}
If $\Phi$ is a shift-invariant nearest-neighbor interaction, $\mu$
is an equilibrium state for $A_{\Phi}$, and $I_\mu$ is continuous,
then $\mu$ is also an equilibrium state for $-I_{\mu}$.
\end{lemma}

\begin{proof}

Let $X^\PP = \{x(\PP): x \in X\}$.

Since $\log$ is a concave function, we can use Jensen's inequality to see that for any shift-invariant measure $\nu$ on $X$,
\begin{multline*}
\int_X I_{\nu} - I_{\mu} \ d\nu =
\int_X \log \left(\frac{\mu(x(0) \ | \ x(\mathcal{P}))}{\nu(x(0) \ | \ x(\mathcal{P}))} \right) \ d\nu
\leq \log \int_X \frac{\mu(x(0) \ | \ x(\mathcal{P}))}{\nu(x(0) \ | \ x(\mathcal{P}))} \ d\nu\\
= \log \int_{X^\PP} \sum_{a \in \A} \frac{\mu(x(0) = a \ | \ x(\mathcal{P}))}{\nu(x(0) = a \ | \ x(\mathcal{P}))} \ \nu(x(0) = a \ | \ x(\mathcal{P})) \ d\nu(\mathcal{P})\\
= \log \int_{X^\PP} \sum_{a \in \A} \mu(x(0) = a \ | \ x(\mathcal{P}) \ d\nu(\mathcal{P})) =
\log \int_{X^\PP} \ d\nu(\mathcal{P}) = \log 1 = 0.
\end{multline*}

So, $\int I_{\mu} \ d\nu \geq \int I_{\nu} \ d\nu$. For any $\nu$,
we have $h(\nu) = \int I_{\nu} \ d\nu$. Therefore,
\[
h(\nu) - \int I_{\mu} \ d\nu \leq h(\nu) - \int I_{\nu} \ d\nu = 0 = h(\mu) - \int I_{\mu} \ d\mu.
\]

This means that the function $h(\rho) + \int -I_{\mu} \ d\rho$ is
maximized at $\rho = \mu$, so $\mu$ is an equilibrium state for
$-I_{\mu}$ by definition.

\end{proof}

\begin{corollary}\label{easycor2}

If $\Phi$ is a nearest-neighbor interaction with underlying SFT $X
= X_\Phi$, $\mu$ is a Gibbs measure for $\Phi$,\\

\noindent {\rm (E1)} $X$ satisfies the classical D-condition, and

\noindent {\rm (E2)} $I_{\mu} = A_{\Phi'}$ for some absolutely
summable interaction $\Phi'$ with $X_{\Phi'} = X$,\\

\noindent then
$$
P(\Phi) = \int I_{\mu}(x) + A_{\Phi}(x) \ d\nu =
\int I_{\mu}(x) - \sum_{i=1}^d \Phi(x(\{0,e_i\})) \ d\nu
$$
for every shift-invariant measure $\nu$  with $\supp(\nu) \subseteq
X$.

\end{corollary}

\begin{proof}
Since $X$ satisfies the $D$-condition, $\mu$ is an equilibrium state
for $A_\Phi$. By Lemma~\ref{ruellelem}, $A_\Phi$ and $-I_{\mu}$ are
physically equivalent. Since $-I_{\mu} = A_{-\Phi'}$ for an
absolutely summable interaction $\Phi'$ and the D-condition holds,
by Theorem~\ref{ruellethm}, $\int A_{\Phi} + I_{\mu} \ d\nu$ is a
constant over all shift-invariant measures $\nu$ on $X$. But, for
$\nu = \mu$, this integral is $\int A_{\Phi} + I_{\mu} \ d\mu =
P(A_{\Phi}) = P(\Phi)$. Therefore, $P(\Phi) = \int I_{\mu} + A_{\Phi} \
d\nu$ for all shift-invariant $\nu$ on $X$.
\end{proof}

Corollary~\ref{easycor4} and Corollary~\ref{easycor2} give the same
integral representation for $P(\Phi)$ for every shift-invariant
measure $\nu$. The classical D-condition is assumed for both, but
the other hypotheses relate to different types of continuity: (D2)
and (D3) of Corollary~\ref{easycor4} imply past-continuity of
$I_\mu$ (by Proposition~\ref{pos}), while (E2) of
Corollary~\ref{easycor2} is a strong form of continuity of $I_\mu$:
as mentioned earlier, Ruelle~\cite[Section 3.2]{Ruelle} showed that any continuous function can be realized as $A_{\Phi'}$ for some summable (but not necessarily absolutely summable) interaction $\Phi'$.
However, we do not know if either
Corollary implies the other.

One might ask for conditions that are sufficient for each of the corollaries to apply.
For Corollary~\ref{easycor4},  we claim that SSF and SSM suffice. Of course, SSF implies (D1) and (D3)
by Propositions~\ref{SSF_D} and~\ref{SSF_positive}.
And SSM implies  uniform convergence of the sequence $\hat{p}^n_\mu$ introduced in the proof of
Proposition~\ref{SSM_continuity}; this in turn implies uniform convergence of the sequence $\hat{p}_{\mu,n}$,
which implies (D2) by Proposition~\ref{unif_conv_cts}(i).
For Corollary~\ref{easycor2}, we claim that  SSF and SSM rate at
sufficiently high rate suffice.
Of course, SSF implies (E1). For (E2), first observe that SSM at rate $f(n)$
implies that
$$
a_n:= \sup_{x,y \in X: x(B_n) = y(B_n)} |p_\mu(x) - p_\mu(y)|
$$
also decays at rate $f(n)$; since  SSF implies that  $p_\mu(x)$ is bounded away from zero on $\supp(\mu)$, the sequence
$$
b_n := \sup_{x,y \in X: x(B_n) = y(B_n)} |I_\mu(x) - I_\mu(y)|
$$
decays at rate  $C'f(n)$ for some $C' > 0$.
One can check that $\sum_n n^d b_n < \infty$
is sufficient to guarantee a realization of
$I_{\mu}$ as $A_{\Phi'}$ for some absolutely summable interaction $\Phi'$
by Ruelle's realization procedure~\cite[Section 3.2]{Ruelle}; so,
$f(n) =  Cn^{-d-2}$ for some $C > 0$ will do.

\

\section*{Acknowledgments}\label{acks}
We are grateful to Raimundo Brice\~{n}o, Nishant Chandgotia and David Gamarnik for several
helpful discussions. We also thank the anonymous referee for several suggestions which improved this paper.

\bibliographystyle{plain}
\bibliography{pressure}

\end{document}